\documentclass[sn-mathphys,Numbered]{sn-jnl}

\usepackage{graphicx}%
\usepackage{multirow}%
\usepackage{amsmath,amssymb,amsfonts}%
\usepackage{amsthm}%
\usepackage{mathrsfs}%
\usepackage[title]{appendix}%
\usepackage{xcolor}%
\usepackage{textcomp}%
\usepackage{manyfoot}%
\usepackage{booktabs}%
\usepackage{algorithm}%
\usepackage{algorithmicx}%
\usepackage{algpseudocode}%
\usepackage{listings}%
\usepackage{natbib}

\def\bfx{{\bf x}}

\theoremstyle{thmstyleone}%
\newtheorem{theorem}{Theorem}
\theoremstyle{thmstyletwo}%
\newtheorem{remark}{Remark}%
\newtheorem{lemma}{Lemma}
\theoremstyle{thmstylethree}%

\begin{document}

\title[The Optimal Linear B-splines Approximation via Kolmogorov Superposition Theorem and its Application]{The Optimal Linear B-splines Approximation via Kolmogorov Superposition Theorem and its Application}


\author[1]{\fnm{Ming-Jun} \sur{Lai}}\email{mjlai@uga.edu}

\author*[2]{\fnm{Zhaiming} \sur{Shen}}\email{zshen49@gatech.edu}


\affil[1]{\orgdiv{Department of Mathematics}, \orgname{University of Georgia}, 
\orgaddress{ \city{Athens}, \postcode{30602}, \state{Georgia}, \country{U.S.A}}}

\affil[2]{\orgdiv{School of Mathematics}, \orgname{Georgia Institute of Technology}, \orgaddress{\city{Atlanta}, \postcode{30332}, \state{Georgia}, \country{U.S.A}}}



\abstract{We propose a new approach for approximating functions in $C([0,1]^d)$ via Kolmogorov superposition theorem (KST) based on the linear spline interpolation of the outer function in the Kolmogorov representation. We improve the results in \cite{LaiShenKST21} by showing that the optimal rate of approximation based on our proposed approach is $O(\frac{1}{n^2})$, where $n$ denotes the number of knots over $[0,1]$. Furthermore, the approximation constant scales linearly with the dimension $d$. 
We show that there exists a dense subclass in $C([0,1]^d)$ whose approximation can achieve such optimal rate,
and the number of parameters needed in such approximation is at most $O(nd)$. Thus, there is no curse of dimensionality when approximating functions in this subclass. Moreover, for $d\geq 4$, we apply tensor product spline denoising technique to denoise KB-splines and get the smooth LKB-splines. We use LKB-splines as basis to approximate functions for the cases when $d=4$ and $d=6$, which extends the results in \cite{LaiShenKST21}. In addition, we validate via numerical experiments that fewer than $O(nd)$ function values are needed to achieve the rate  $O(\frac{1}{n^\beta})$ for some $\beta>0$
based on the smoothness of the outer function. Finally, we demonstrate that our approach can be applied to numerically solving partial differential equation such as the Poisson equation with accurate results. }




\maketitle

\section{Introduction}\label{intro}

Fast and effective computation of high dimensional function approximation has been at the research frontier since the advent of deep neural network. The primary challenge lies in overcoming the curse of dimensionality, a longstanding computational bottleneck. Despite decades of effort, progress has been limited, with only a few notable advances. One such breakthrough is the introduction of the Barron class of functions, as proposed in \cite{Barron93, Barron94}. It is also well explained in \cite{Pinkus99}. The most recent results on neural network approximation of this class of functions can be found in \cite{Siegel20} and \cite{Klusowski18}. 

Recently, the authors of \cite{LaiShenKST21} demonstrated that certain subclasses of functions, which are dense in $C([0,1]^d)$, can be effectively approximated by applying Kolmogorov superposition theorem without suffering from the curse of dimensionality. In this work, we shift our focus to exploring this promising direction. Let us first recall the Kolmogorov superposition theorem (KST), we introduce two versions of KST which appear in \cite{Kolmogorov57} and \cite{Lorentz66}, respectively.

\begin{theorem}[Kolmogorov Superposition Theorem -- original version \cite{Kolmogorov57}] 
\label{kolmogorovOriginal}
Let $f\in C([0,1]^d)$, then there exist continuous functions $g_q:\mathbb{R}\to\mathbb{R}$ and $\phi_{qp}:[0,1]\to\mathbb{R}$ such that
\begin{equation}
\label{represoriginal}
f(x_1,\cdots,x_d)=\sum_{q=0}^{2d} g_q\left(\sum_{p=1}^d\phi_{qp}(x_p)\right).
\end{equation}
\end{theorem}
The significance of this surprising result can be summarized succinctly: Every continuous multivariate function can be obtained from univariate continuous functions using compositions and additions. There have been many improvements of KST over the years. Lorentz \cite{Lorentz62} pointed out that the outer function $g_q$ can be chosen to be the same, while Sprecher \cite{Sprecher63} showed that one can take $\phi_{qp}=\lambda_p\phi_q$. Henkin \cite{Henkin64} and Fridman \cite{Fridman67}, respectively, pointed out that the inner functions $\phi_{qp}$ can be chosen to be Hölder continuous with exponent $\alpha\in (0,1)$ and Lipschitz continuous. 
Sprecher \cite{Sprecher65a, Sprecher65b, Sprecher66, Sprecher72} also showed that inner functions can be replaced by one single inner function with an appropriate shift in its argument through the constructive form of KST. An excellent explanation of the history about the development of KST can be found in \cite{Morris21}. We now turn our attention to the Lorentz's version of KST \cite{Lorentz66}, which is more useful for the development of our approach.

\begin{theorem}[Kolmogorov Superposition Theorem -- Lorentz's version \cite{Lorentz66}] 
\label{kolmogorov}
There exist $0<\lambda_p\leq 1$, $p=1,\cdots, d$, and strictly increasing $\alpha$-Hölder continuous functions $\phi_q(x): [0,1] \to [0,1]$, $q=0, \cdots, 2d$, with exponent $\alpha\in (0,1)$, such that for every $f\in C([0,1]^d)$, 
there exists a continuous function $g\in C([0,d])$, such that 
\begin{equation}
\label{repres}
f(x_1,\cdots,x_d)=\sum_{q=0}^{2d} g\left(\sum_{p=1}^d\lambda_p\phi_q(x_p)\right).
\end{equation}
\end{theorem}

\begin{remark}
    The exponent $\alpha$ in the Lorentz's version of KST can be chosen as $\alpha=\log_{4d+2} 2$ according to the construction in \cite{Lorentz66} and \cite{bryant2008kolmogorov}. We follow their construction by using the same $\alpha$ for our implementation of functions $\phi_q$ numerically. In general, $\alpha$ can be chosen as any value between $0<\alpha<1$ independently of dimension $d$.
\end{remark}

Some notable features of the representation formula (\ref{repres}) are the following. Firstly, there is only one outer function $g$ associated with $f$. Secondly, the number $2d+1$ of summands can not be further reduced \cite{Ostrand65,Sternfeld85}. Thirdly, the inner functions cannot be chosen to be continuously differentiable \cite{Vitushkin54,Vitushkin64,Lorentz62}. The upshot for this representation is: for any continuous function $f\in C([0, 1]^d)$, there is a continuous function $g_f\in C([0, d])$ so that $f$ can be represented by $g_f$ via (\ref{repres}). Conversely, given any continuous function $g\in C([0, d])$, we can produce a continuous function 
$f_g \in C([0,1]^d)$ by using the representation formula (\ref{repres}). Such correspondence between $f$ and $g$ is one-to-one. Therefore we can use what we understand about univariate continuous functions to understand multivariate continuous functions. 

It is worth noting that KST also has some useful topology and machine learning interpretations. KST essentially establishes that all $d$ dimensional compact metrizable spaces can be embedded into $\mathbb{R}^N$ if and only if $N\geq 2d+1$. KST also guarantees that any continuous statistical or machine learning model, after a suitable embedding, is a sum of generalized additive models. There have been many generalizations and extensions of KST over the past few decades. Ostrand \cite{Ostrand65} showed that KST holds on compact metric spaces. Doss \cite{Doss77} and Demko \cite{Demko77} extended KST to $\mathbb{R}^n$ for unbounded and bounded continuous functions, respectively. Feng \cite{Feng10} generalized KST to locally compact and finite dimensional separable metric spaces. 

It is straightforward to see that the representation formula (\ref{repres}) mimics the structure of a two-layer neural network where the inner and outer functions can be considered as activation functions. However, there have been debates for decades on whether such a representation via KST is useful. Girosi and Poggio \cite{Girosi1989} claimed that some degree of smoothness is required for inner and outer functions in order for the approximation to generalize and stabilize against noise. Lin and Unbehauen \cite{Lin1993} made a similar conclusion by noting that all information carried by $f$ must be contained in the univariate function $g$ hence learning the latter is not any easier than learning the former. On the other hand, Køurkovà \cite{Kurkova91,Kurkova92} countered some of the criticisms from Girosi and Poggio by giving a constructive way to approximate the univariate outer function $g$ through linear combinations of the smooth sigmoid function. They also bounded the number of units needed for a desired approximation. This has in turn generated further interest in the study of neural network and approximation.

Indeed, KST has been actively studied which echoes the fast development of neural network computing \cite{Cybenko89,MM92,Pinkus99}. Hecht-Nielsen \cite{H87} was among the first to draw a connection between KST and neural networks. This inspired much of the later works on universality of two-layer neural networks. However, Hecht-Nielsen was doubtful about the direct usefulness of this connection because no construction of the outer function was known then and he mentioned the possibility of learning the outer function from input-output examples. Later on, Igelnik and Parikh \cite{IP03} proposed a neural network algorithm using spline functions to approximate both the inner and outer functions. 

More recently, active research has been conducted on neural network approximation via KST and achieves promising results \cite{Maiorov99,Guliyev2018,Montanelli2020,Schmidt2021,fakhoury2022exsplinet}. However, these results are not directly based on the representation formula (\ref{kolmogorov}) and can be impractical to implement. The authors in \cite{liu2024kan} proposed Kolmogorov-Arnold Networks (KAN), in which the activation functions become learnable rather than the weights in the traditional feed-forward neural networks. While this innovation has demonstrated promising numerical results in certain experiments, it remains unclear to what extent the new approach can fundamentally outperform the traditional feed-forward neural networks.

Besides all of these, the authors in \cite{LaiShenKST21} introduced a class called Kolmogorov-Lipschitz (KL) continuous functions and proposed LKB-splines for approximating such functions with the rate $\mathcal{O}(1/n)$ and complexity $\mathcal{O}(dn)$. Note that LKB-splines are a smooth version of KB-splines and the KB-splines are similar to the Kolmogorov spline network in the literature (cf. \cite{IP03}). These KB-splines are very noisy and deemed not useful at all in practice. One of the significant features 
of the work \cite{LaiShenKST21} is the denoising of KB-splines to get LKB-splines in $\mathbb{R}^2$ 
or $\mathbb{R}^3$, which are bivariate or trivariate spline
functions (cf. \cite{LS07}) after a denoising technique based on penalized least squares method (cf. \cite{LS09}). In this paper, we would like to follow up along these directions. 


The remaining part of this paper is structured as follows. In Section~\ref{KHolder}, we introduce a subclass of continuous functions, called Kolmogorov-Hölder (KH) class. We show that this class is dense in $C([0,1]^d)$ by showing a subclass of this class is dense. We also introduce the linear KB-spline functions based on the linear interpolation of the outer function and show that there is a dense subclass of $C([0,1]^d)$ which can be approximated by using KB-splines with the optimal rate $\mathcal{O}(1/n^2)$. In Section~\ref{TD}, we introduce a tensor product spline denoising method to smooth the KB-spline basis and get the corresponding LKB-splines as basis for our approximation scheme. In section \ref{experiments}, we demonstrate the numerical results for function approximation in $d=4$ and $d=6$ by using linear LKB-splines as basis. In Section~\ref{secPDE}, we show the numerical method of solving Poisson equation based on LKB-splines as one application of our approach. Finally, in Section~\ref{secCon}, we conclude the paper and point out some future research directions.

\section{Kolmogorov-Hölder class} \label{KHolder}
We will consider a general class of continuous functions called Kolmogorov-Hölder (KH) class. Let us call the functions $g$ and $\phi_q$, $q=0,\cdots,2d$, in (\ref{kolmogorov}) the outer function and inner functions respectively.
Suppose $\beta\in (0,1)$, for each  function $f\in C([0, 1]^d)$, we define 
\begin{equation}
\label{G1}
\text{KH}_{\beta}:=\{f:  \hbox{ the outer function $g$ is $\beta$-Hölder continuous} \}
\end{equation}
to be the class of Kolmogorov-Hölder continuous functions with exponent $\beta$. Recall that we say a function $f:[0,1]^d\to\mathbb{R}$ is in $C^{0,\alpha}([0,1]^d)$ if 
\begin{equation}
\label{Lipalpha}
  \sup_{\mathbf{x},\mathbf{y}\in [0,1]^d}  \frac{|f(\mathbf{x})-f(\mathbf{y})|}{\|\mathbf{x}- \mathbf{y}\|^\alpha} <\infty. 
\end{equation}
One can show 
\begin{theorem}
Let $f\in C([0,1]^d)$ has the KST representation (\ref{repres}). Suppose $f\in \text{KH}_{\beta}$ for some $\beta\in (0,1)$, then $f\in C^{0,\alpha\beta} ([0,1]^d)$ with $\alpha$ being the Hölder exponent for the inner function $\phi_q$ via representation (\ref{repres}).
\end{theorem}
\begin{proof}
Let $g$ and $\phi_q$, $q=0,\cdots,2d$, be the functions as defined in (\ref{repres}).
Suppose $\mathbf{x},\mathbf{y}\in [0,1]^d$, and $(x_1,\cdots,x_d)=\mathbf{x}\neq\mathbf{y}=(y_1,\cdots, y_d)$. Then
    \begin{align*}
        \left|f(\mathbf{x}) - f(\mathbf{y})\right|&=\left|\sum_{q=0}^{2d}g\left(\sum_{i=1}^d \lambda_i\phi_q(x_i)\right)-\sum_{q=0}^{2d}g\left(\sum_{i=1}^d \lambda_i\phi_q(y_i)\right)\right| \\
        & \leq \sum_{q=0}^{2d}\left|g\left(\sum_{i=1}^d \lambda_i\phi_q(x_i)\right)-g\left(\sum_{i=1}^d \lambda_i\phi_q(y_i)\right)\right| \\
        & \leq \sum_{q=0}^{2d}C_1\left|\sum_{i=1}^d \lambda_i\phi_q(x_i)-\sum_{i=1}^d \lambda_i\phi_q(y_i)\right|^{\beta} \leq\sum_{q=0}^{2d}C_1\sum_{i=1}^d\lambda_i^{\beta}\left|\phi_q(x_i)-\phi_q(y_i)\right|^{\beta}  \\
        & \leq \sum_{q=0}^{2d}C_1\sum_{i=1}^d\lambda_i^{\beta}C_2^{\beta}|x_i-y_i|^{\alpha\beta} \leq (2d+1)C_1 C_2^{\beta}\sum_{i=1}^d|x_i-y_i|^{\alpha\beta},
    \end{align*}
for some constants $C_1,C_2>0$.
This completes the proof. 
\end{proof}

Next, let us introduce two important subclasses of KH function class: Kolmogorov-polynomials and Kolmogorov B-splines (KB-splines).

\subsection{Kolmogorov-polynomials}
Let us define the Kolmogorov-polynomial as
\begin{equation}
Kp_n(x_1,\cdots,x_d)=\sum_{q=0}^{2d}p_n\left(\sum_{i=1}^d\lambda_i\phi_q(x_i)\right),
\end{equation}
where the function $p_n$ is a univariate polynomial. We call it a Kolmogorov-monomial if $p_n(t):=t^n$, $n\geq 0$. Figure~\ref{Kpolynomials} shows some plots of different Kolmogorov-monomials with and without using the denoising/smoothing technique described in \cite{LaiShenKST21}. {In those plots, we fix the scalars $\lambda_1=1$, $\lambda_2=1/\sqrt{2}$ in the representation (\ref{repres}), and we implement the inner functions $\phi_q(x), q=0,1,2,3,4$, the same way as described in \cite{Lorentz66} and \cite{bryant2008kolmogorov}.}  The significance of Kolmogorov-monomials is that the $span\{Kp_n\}_{n\geq 0}$ is dense in $C([0,1]^d)$. Let us call this result the Kolmogorov-Weierstrass theorem.

\begin{figure}[t] 
    \centering
    \begin{tabular}{cc}
    \includegraphics[width=0.45\textwidth]{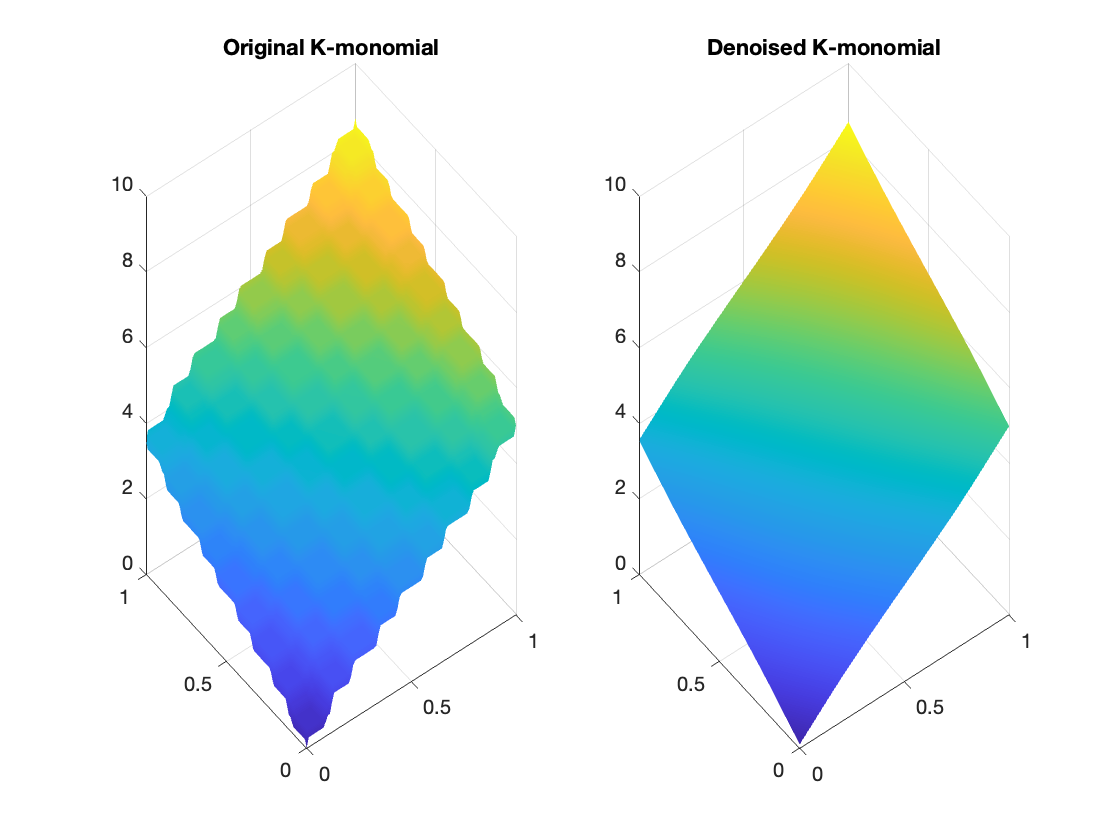} &
    \includegraphics[width=0.45\textwidth]{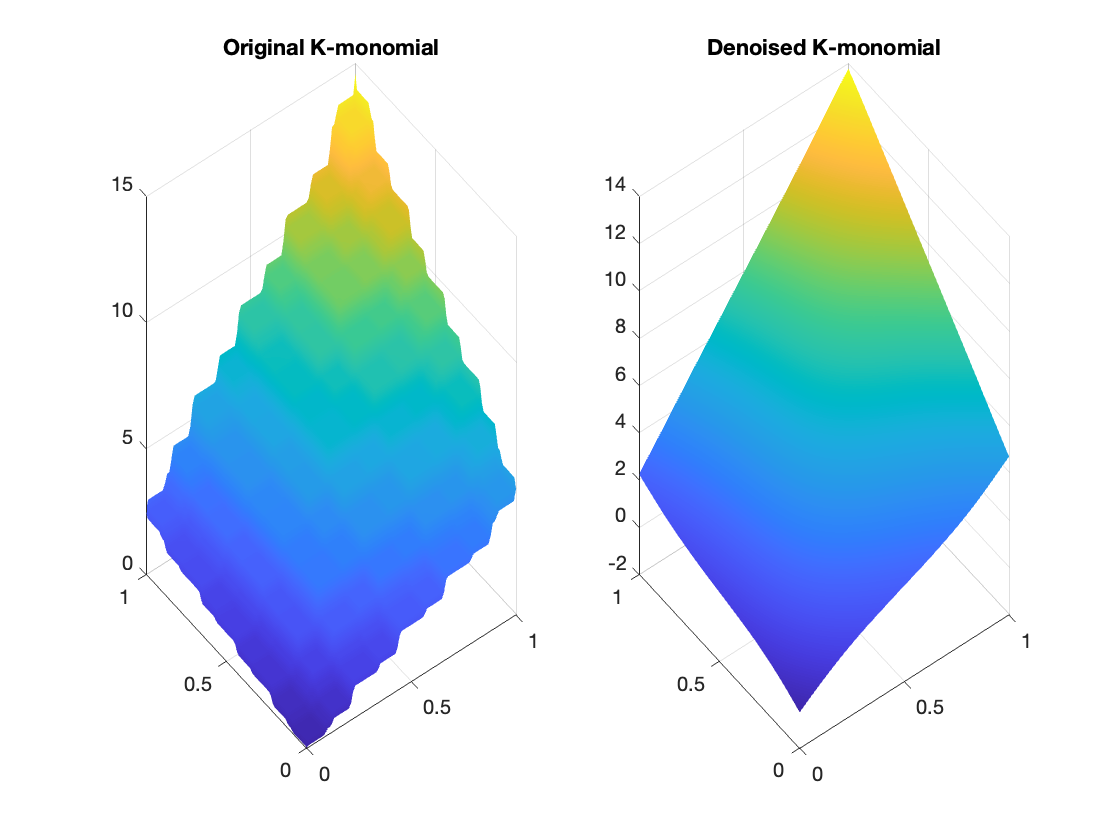} \\
    \includegraphics[width=0.45\textwidth]{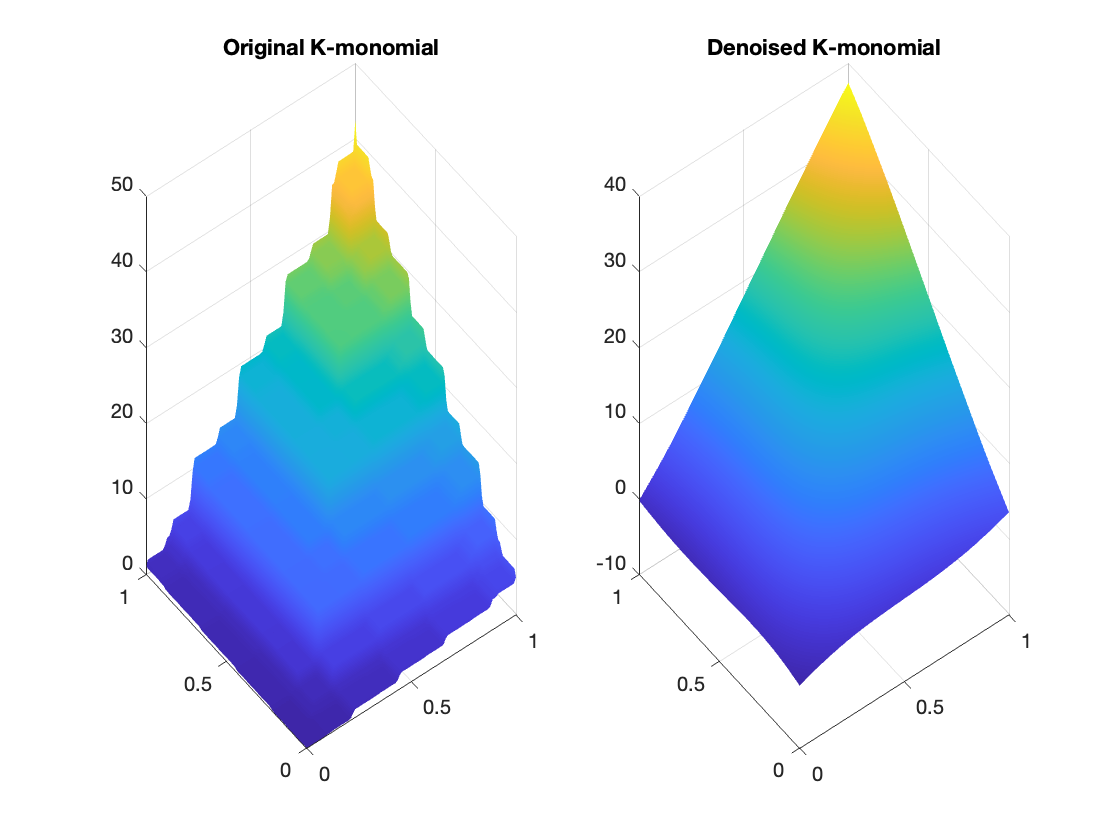} &
    \includegraphics[width=0.45\textwidth]{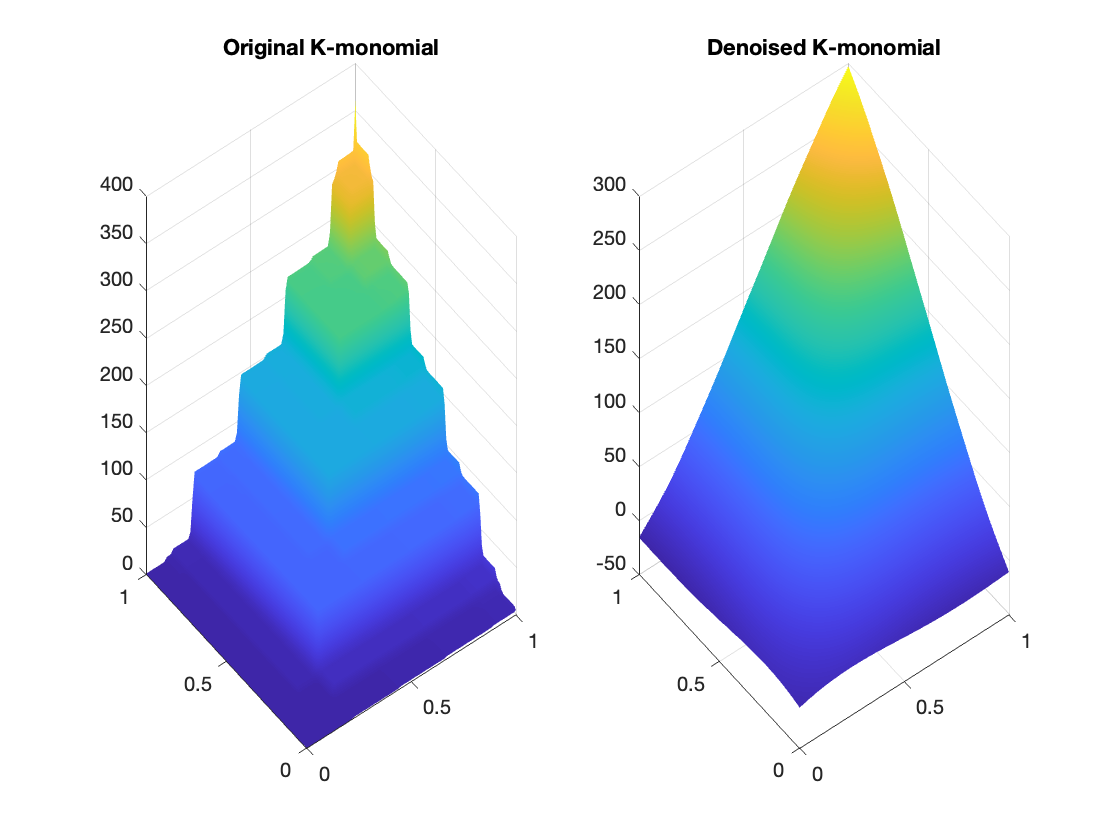} 
    \end{tabular}
    
    \caption{Examples of Kolmogorov-monomials (Top Row: $p_n(x)=x, x^2$. Bottom Row: $p_n(x)=x^4, x^8$). \label{Kpolynomials}}
\end{figure} 

\begin{theorem} [Kolmogorov-Weierstrass Theorem] \label{KW}
For any $f\in C([0,1]^d)$ and any $\epsilon>0$, there exists $K\in span\{Kp_n\}_{n\geq 0}$ with $p_n(t)=t^n$ such that 
\begin{equation}
    \|f-K\|_{\infty}\leq \epsilon.
\end{equation}
\end{theorem}
\begin{proof}
    By Kolmogorov superposition theorem, we can write $f(x_1,\cdots,x_d)=\sum_{q=0}^{2d}g\left(\sum_{i=1}^d\lambda_i\phi_q(x_i)\right)$. By Weierstrass approximation theorem, there exists a polynomial $p(t)$ such that $|p(t)-g(t)|\leq\frac{\epsilon}{2d+1}$ for all $t\in [0,d]$. By letting \[K(x_1,\cdots,x_d)=\sum_{q=0}^{2d}p\left(\sum_{i=1}^d\lambda_i\phi_q(x_i)\right)\in span\{Kp_n\}_{n\geq 0},\] we get
\begin{align*}
    |f(x_1,\cdots,x_d)-K(x_1,\cdots,x_d)|&=\left|\sum_{q=0}^{2d}g\left(\sum_{i=1}^d\lambda_i\phi_q(x_i)\right)-\sum_{q=0}^{2d}p\left(\sum_{i=1}^d\lambda_i\phi_q(x_i)\right)\right|  \\
    &\leq \sum_{q=0}^{2d}\left|g\left(\sum_{i=1}^d\lambda_i\phi_q(x_i)\right)-p\left(\sum_{i=1}^d\lambda_i\phi_q(x_i)\right)\right| \\
    &\leq (2d+1)\cdot \frac{\epsilon}{(2d+1)} =\epsilon
\end{align*}
as desired.
\end{proof}

\begin{remark}
    The Kolmogorov-Hölder continuous function class is very large, in fact dense in $C([0,1]^d)$ by Theorem~\ref{KW}. Indeed, in addition to Kolmogorov-polynomials, we can use trigonmetric functions as outer function $g$ to define high dimensional continuous functions called Kolmogorov-trigonometric functions via (\ref{repres}).  Similarly, we can have Kolmogorov-exponential functions, Kolmogorov-logarithmic functions, etc,. 
     In fact, any univariate Hölder continuous function $g$ gives a Kolmogorov-Hölder continuous function $f$ via Kolmogorov representation formula by using Theorem~\ref{kolmogorov}. Because
     these univariate functions $g$ are H\"older continuous, their corresponding $f$ are in the $\text{KH}_\beta$ class for some $\beta>0$.  
\end{remark}






\subsection{Linear KB-splines and LKB-splines}
It is well known that linear spline function can be represented in terms of linear 
combinations of ReLU functions and vice versa, see, e.g. \cite{DD19}, and \cite{DHP21}.
Let $S^0_1(\triangle)$ be the space of all continuous linear splines over the partition
$\triangle=\{0=t_0<t_1<\cdots <t_n=1\}$ with $|\triangle|=\max_i |t_i-t_{i-1}|$. For univariate function $f$, let $\omega(f,h):=\sup_{|x-y|\leq h}|f(x)-f(y)|$ be its modulus of continuity.
From standard approximation theory (c.f. \cite{P81}), we know that

\begin{lemma} \label{Linearspline}
    Suppose $f\in C([0,1])$, let $\triangle$ be a partition over $[0,1]$ with $n$ knots. Then there exists a $L_f\in S^0_1(\triangle)$ such that
    \begin{equation}
    \|f-L_f\|_{\infty}\leq
\begin{cases}
  \omega(f, \frac{1}{n}), \quad \text{if} \quad f\in C([0,1]),\\      
  \frac{1}{2n}\|f'\|_{\infty}, \quad \text{if} \quad f\in C^1([0,1]),\\
  \frac{1}{8n^2}\|f''\|_{\infty}, \quad \text{if} \quad f\in C^2([0,1]).
\end{cases}
\end{equation}
\end{lemma}

\begin{remark} \label{optimalratermk}
Note that even if we can further increase the smoothness of function $f$, the approximation rate is not getting better. In order to achieve a better approximation rate for those $f$ with higher order smoothness, one has to use a higher degree splines.
Therefore, for linear spline approximation, $\mathcal{O}(1/n^2)$ is the optimal approximation rate.
\end{remark}

For $f\in C([0,1]^d)$, we would like to apply Lemma~\ref{Linearspline} for approximating the outer function $g$, and hence approximating $f$ via the representation formula (\ref{repres}). For this purpose, let us define the linear KB-splines of $f$ as 
\begin{equation}
KB(f)_n(x_1,\cdots,x_d):=\sum_{q=0}^{2d}L_g\left(\sum_{i=1}^d\lambda_i\phi_q(x_i)\right),
\end{equation}
where $L_g$ is chosen to be the linear spline interpolation of the outer function $g\in C([0,d])$ with uniform partition of $[0,d]$ with $nd$ knots, i.e., $|\triangle|=\frac{1}{n}$. Then by Theorem \ref{kolmogorov} and Lemma \ref{Linearspline}, we have
\begin{theorem} \label{OptimalRate}
Suppose $f\in C([0,1]^d)$. Then
\begin{equation}
    \|f-KB(f)_n\|_{\infty}\leq
\begin{cases}
  (2d+1)\omega(g, \frac{1}{n}), \quad \text{if} \quad g\in C([0,d]),\\      
  \frac{2d+1}{2n}\|g'\|_{\infty}, \quad \text{if} \quad g\in C^1([0,d]),\\
  \frac{2d+1}{8n^2}\|g''\|_{\infty}, \quad \text{if} \quad g\in C^2([0,d]).
\end{cases}
\end{equation}
\end{theorem}

\begin{proof}
Let us show only the proof for the case $g\in C^2([0,d])$, the proofs for the other two cases are similar. For any $\mathbf{x}=(x_1,\cdots,x_d)$, we have
\begin{align*}
    \left|f(\mathbf{x})-KB(f)_n(\mathbf{x})\right|&=\left|\sum_{q=0}^{2d}g\left(\sum_{i=1}^d\lambda_i\phi_q(x_i)\right)-\sum_{q=0}^{2d}L_g\left(\sum_{i=1}^d\lambda_i\phi_q(x_i)\right)\right|  \\
    &\leq \sum_{q=0}^{2d}\left|g\left(\sum_{i=1}^d\lambda_i\phi_q(x_i)\right)-L_g\left(\sum_{i=1}^d\lambda_i\phi_q(x_i)\right)\right| 
    \leq \frac{2d+1}{8n^2}\|g''\|_{\infty}
\end{align*}
as desired.
\end{proof}

Theorem \ref{OptimalRate} immediately shows linear KB-splines are dense in $C([0,1]^d)$. More importantly, {the approximation rate of linear KB-splines is independent of dimension $d$} while the approximation constant is linearly dependent on $d$.  Thus, we conclude that 
the approximation of high dimensional continuous function $f$ does not suffer from the curse of dimensionality  for a subclass of $C([0,1]^d)$, i.e., those $f$ whose outer function $g\in C^1([0,d])$ or $g\in C^2([0,d])$, such a subclass is dense as $C^1([0,d])$ and $C^2([0,d])$ are dense in $C([0,d])$. In fact, there are many choices of such $g$. For example, $g$ can be polynomial functions, trigonometric functions, exponential functions, etc,. Moreover, as discussed in Remark~\ref{optimalratermk}, the optimal rate of approximation for $f$ by using KB-splines is $O(1/n^2)$, which is achieved when the outer function $g_f\in C^2([0,d])$. It is also not hard to see that the number of parameters needed in such approximation equals to the number of knots, which is $O(nd)$.

Let us recall the linear KB-spline basis functions defined in \cite{LaiShenKST21}.
Let $\triangle_n=\{ 0= t_1<t_2<\cdots <t_{dn}<d\}$ 
be a uniform partition of interval $[0, d]$, and $b_{n,i}(t), i=1, \cdots, dn$ 
be the standard univariate linear B-splines, we define the linear KB-spline (basis) functions as
\begin{equation}
\label{KB}
\text{KB}_{n,j}(x_1,\cdots, x_d) := \sum_{q=0}^{2d} b_{n,j}\left(\sum_{i=1}^d\lambda_i\phi_q(x_i)\right), 
j=1, \cdots, dn.
\end{equation}

We showed in \cite{LaiShenKST21} that these $\text{KB}_{n,j}$ have several useful properties, e.g. 
the partition of unity, linear independence, and denseness in $C([0,1]^d)$.
Thus, we can treat $\text{KB}_{n,j}$ as basis functions for approximating $f\in C([0,1]^d)$. However, the basis $\text{KB}_{n,j}$ are not differentiable and has many jumps, hence can not be directly used for approximating $f$. For $d=2$ and $d=3$, we apply a spline denoising technique as introduced in \cite{LaiShenKST21} to denoise the KB-spline basis and get the corresponding smooth LKB-splines. We will briefly introduce the denoising procedure for constructing LKB-splines in the next section. For dimension $d\geq 4$, we need to apply tensor product of such denoising technique, and we will explain it in the next section as well.

\section{Tensor product approximation and denoising}
\label{TD}
Let us first recall the approximation based on tensor product of Bernstein polynomial, which is well-known in the literature. We review them in order
to explain the computation of tensor product splines for denoising in the later subsection. 
\subsection{Tensor product approximation of Bernstein polynomial}

Suppose $f\in C([0,1])$, we define the Bernstein operator of degree $n$ on $f$ as
\begin{equation}
    B_n f(x) := \sum_{i=0}^n f\left(\frac{i}{n}\right)B_{n,i}(x)
\end{equation}
where $B_{n,i}=\binom{n}{i}x^i(1-x)^{n-i}$ is the Bernstein basis polynomial. From standard approximation theory (c.f. \cite{P81}), we know

\begin{lemma} \label{1dBd}
    Suppose $f\in C^2([0,1])$. Then
    \begin{equation}
        \|f-B_nf\|_{\infty}\leq \frac{1}{8n}\|f''\|_{\infty}.
    \end{equation}
 \end{lemma}
\noindent
In general, for $f\in C([0,1]^d)$, we can define  


    
 \begin{equation}
    B_{n_1,\cdots,n_d} f(x_1,\cdots, x_d):=\sum_{i_1=0}^{n_1}\cdots\sum_{i_d=0}^{n_d} f\left(\frac{i_1}{n_1},\cdots,\frac{i_d}{n_d}\right)B_{n_1,i_1}(x_1)\cdots B_{n_d,i_d}(x_d).
\end{equation}
By applying Lemma \ref{1dBd} and 
and a chain of triangle inequalities argument, it is not hard to establish the following result. We leave its proof to the interested readers.
\begin{lemma}
    Suppose $f\in C^2([0,1]^d)$ for integer $d\geq 1$. Then
    \begin{equation}
        \|f-B_{n,\cdots,n}f\|_{\infty}\leq \frac{d}{8n}|f''|_{2,\infty},
    \end{equation}
where
$|f|_{2,\infty}=\max_{i_1+\cdots+i_d=2}\|D_{x_1}^{i_1}\cdots D_{x_d}^{i_d} f\|_{\infty}$.
\end{lemma}

\subsection{Spline denoising} 
The linear KB-splines obtained via (\ref{KB}) are {nonsmooth and have many jumps}, therefore are not directly useful for approximation. We would like to smooth/denoise them so that they will be useful. For self-containedness, let us briefly introduce the ideas of spline denoising and tensor product spline denoising. We leave more details of the denoising procedure to \cite{LaiShenKST21}.

\begin{figure}[t]
    \centering
    \begin{tabular}{cc} 
        \includegraphics[width=0.49\textwidth]{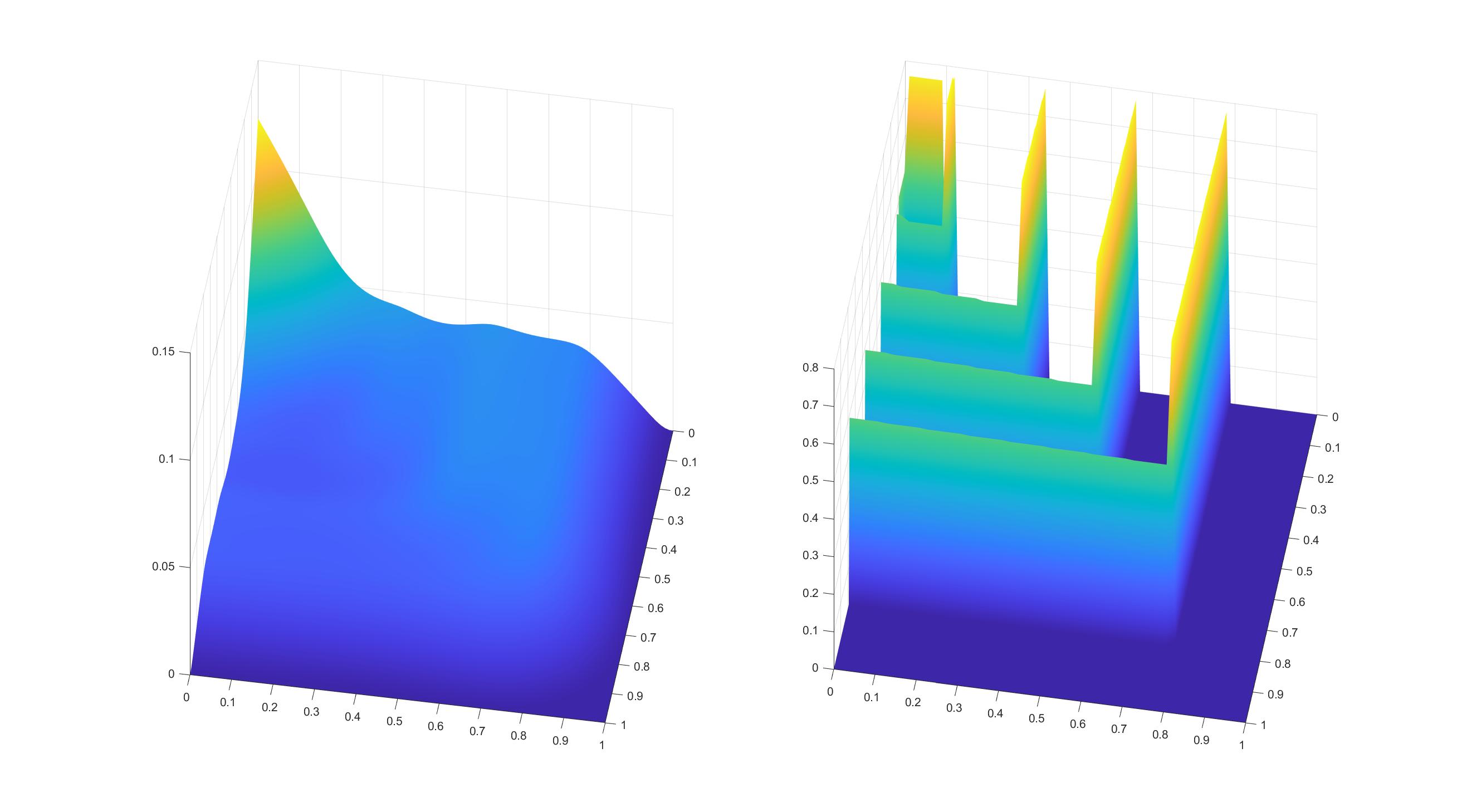}  &  \includegraphics[width=0.49\textwidth]{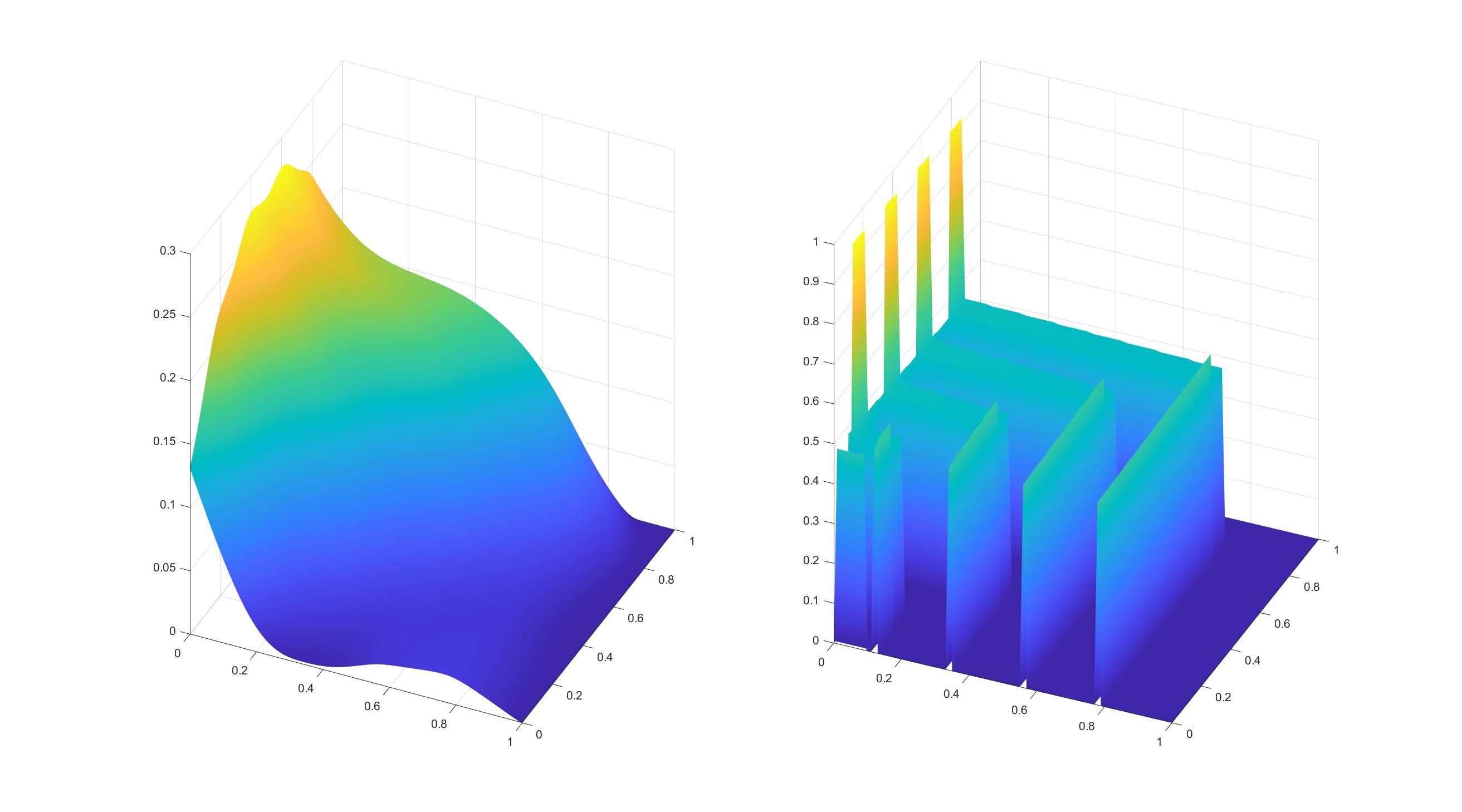} \\
        \includegraphics[width=0.49\textwidth]{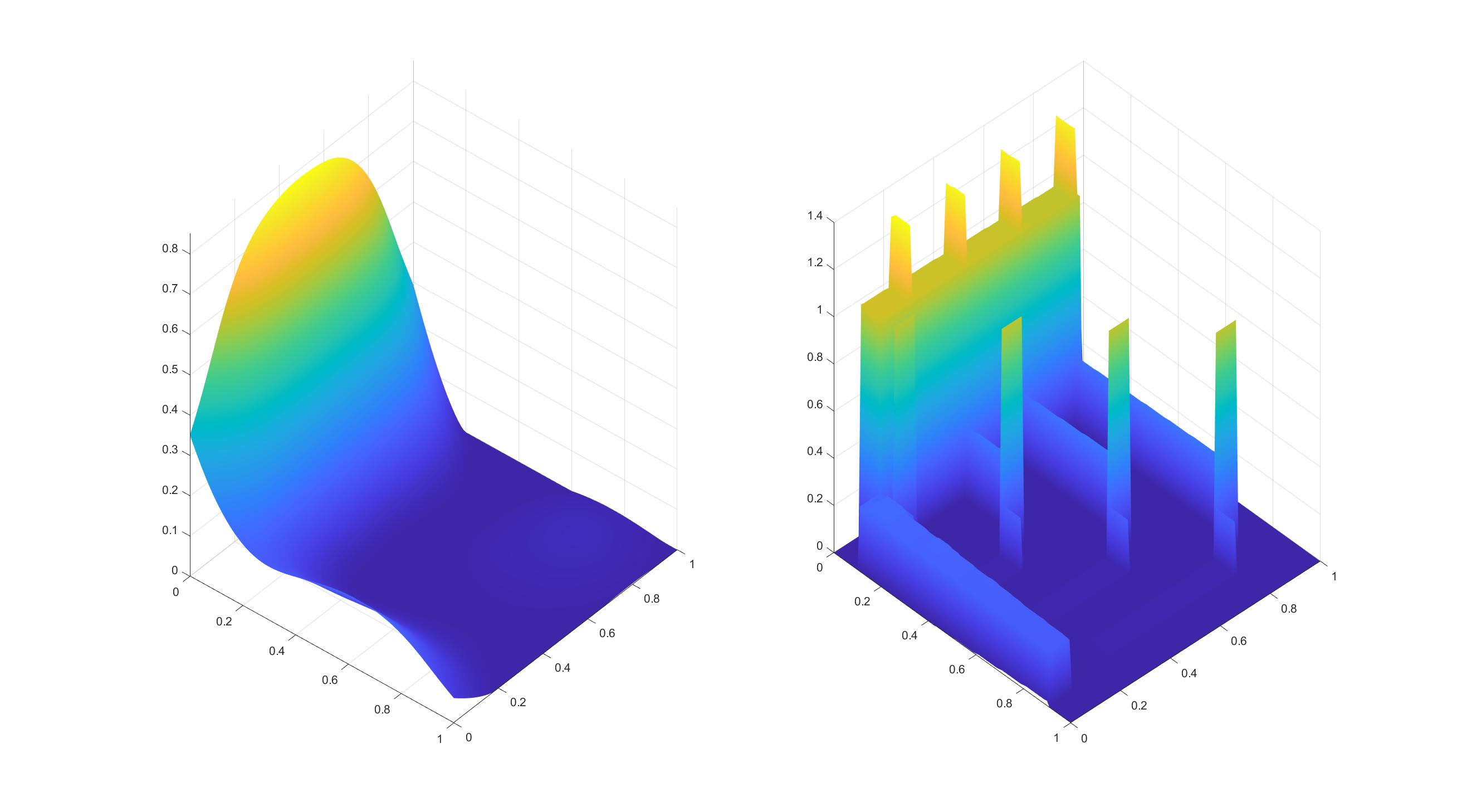} &
        \includegraphics[width=0.49\textwidth]{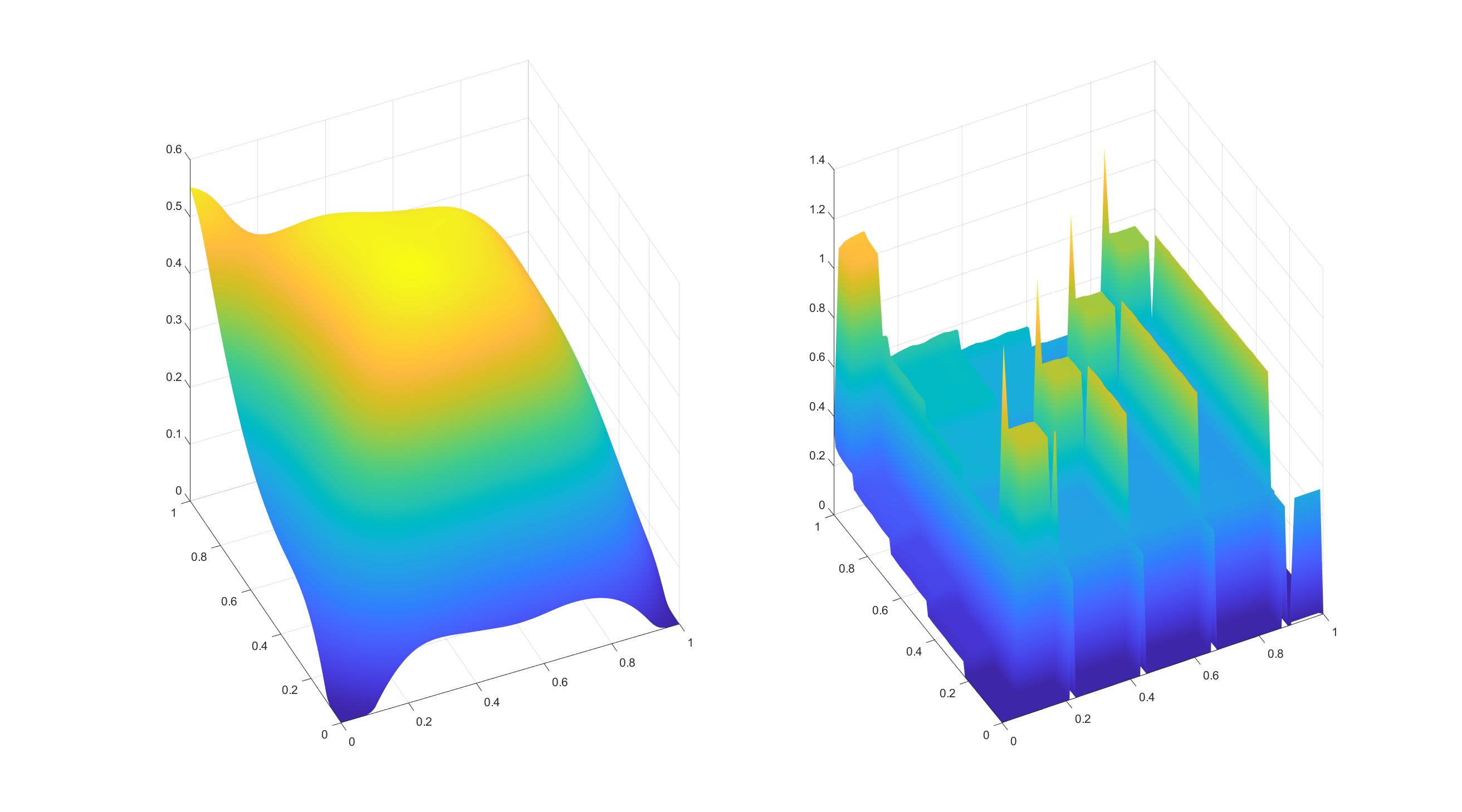} 
    \end{tabular}
    \caption{Some examples of linear LKB-splines (the first and third columns) which are the smoothed version of the corresponding linear KB-splines (the second and fourth columns). \label{f1}}
\end{figure}

For convenience, we base our discussion on the bivariate splines.
Let us first recall bivariate spline space.  For a triangulation $\triangle$ of $[0, 1]^2$, 
for any degree $d\ge 1$ and smoothness $r\ge 1$ with $r<d$, let
\begin{equation}
\label{splinespace}
S^r_d(\triangle)= \{ s\in C^r([0, 1]^2):  s|_T\in \mathbb{P}_d, T\in \triangle\}
\end{equation}
be the spline space of degree $d$ and smoothness $r$ with $d>r$. We refer to \cite{LS07} for 
theoretical details and \cite{ALW06,LL22} and for computational details of multivariate splines.
For a bivariate spline function $s(x,y)\in S^r_d(\triangle)$, we can write it as
\begin{equation} \label{splinecoeff}
s(x,y)=\sum_{i=1}^{m} c_{i} b_{i}(x,y)\in S^r_d(\triangle), 
\end{equation}
where $c_{i}$'s are the spline coefficients, $b_i(\cdot,\cdot)$ are bivariate basis splines with degree $d$ and smoothness $r$, and $m$ is the dimension of the bivariate spline space. Note that the 
computation of $b_i(x,y)$ takes some efforts. We adopt the approach in \cite{ALW06} and {we fix $r=1$ and $d=5$ in our tensor product implementation.}

For a given data set $\{(x_i,y_i, z_i)\}_{i=1}^N$ with data locations
$(x_i,y_i)\in [0, 1]^2$ and noisy data values $z_i=f(x_i,y_i)+\epsilon_i$, $i=1,\cdots,N$.
The penalized least squares method (cf. \cite{L08} and \cite{LS09}) of bivariate spline denoising is to find 
\begin{equation}
\label{PLS2D}
\min_{s\in S^r_d(\triangle)} \sum_{i=1, \cdots, N} \left|s(x_i,y_i)- z_i\right|^2  
+ \lambda {\cal E}_2(s) 
\end{equation}
for some fixed constant $\lambda\approx 1$, where ${\cal E}_2(s)$ is the thin-plate energy functional 
defined as
\begin{equation}
\label{E2}
{\cal E}_2(s) :=\int_{[0,1]^2} \left|\frac{\partial^2}{\partial x^2} s\right|^2 + 2
\left|\frac{\partial^2}{\partial x \partial y} s\right|^2 + \left|\frac{\partial^2}{\partial y^2} s\right|^2. 
\end{equation}  
It is well known that this approach can be used for smoothing noisy data.  In our computation, the triangulation $\triangle$ is 
the one obtained from uniformly refined the initial triangulation $\Delta_0$ three times, where
$\Delta_0$ is  obtained by dividing $[0, 1]^2$ into two triangles using its diagonal line.  

Let us write 
$S^1_5(\triangle)=\hbox{span}\{b_1, \cdots, b_m\}$, and $0= t_1<t_2<\cdots <t_{dn}<d$ 
be a uniform partition of interval $[0, d]$. For each $j=1, \cdots, dn$, we define LKB-splines as
\begin{equation}
    LKB_{n,j}:= \sum_{i=1}^m c_{i,j}b_i
\end{equation}
where the coefficients $c_{i,j}$'s are the solutions to (\ref{PLS2D}) via the representation (\ref{splinecoeff}) with each $z_i$, $i=1,\cdots,N$, in (\ref{PLS2D}) substituted by $KB_{n,j}(x_i,y_i)$. 

The denoised LKB-splines are much smoother and nicer. Some examples of generated linear LKB-splines are shown in Figure~\ref{f1}.

\subsection{Tensor product of spline denoising} \label{TPDenoising}

Now let us explain the idea of tensor product spline based denoising method for smoothing noisy KB-splines. For convenience, 
let us consider the case for $d=4$, similar arguments can be applied to a general $d>4$ by using the tensor product of bivariate and trivariate splines. See \cite{deB79} for the general case for tensor product splines for data 
interpolation. 

For the rest of the discussion, we focus on the tensor product bivariate spline space ${\cal S}:=    S^1_5(\triangle) \times S^1_5(\triangle)$. For a given data set $\{(x_i,y_i, u_j, v_j, z_{i,j}), i,j=1, \cdots, N\}$ with data locations  
$(x_i,y_i,u_j,v_j)\in [0, 1]^2\times [0,1]^2$ and noisy data values $z_{i,j}$,  we can write a spline function
\begin{equation}
s(x,y,u,v)=\sum_{i=1}^{m_1}\sum_{j=1}^{m_2} c_{ij} b_{i}(x,y)b_j(u,v)\in {\cal S}, 
\end{equation}
where $c_{ij}$'s are the spline coefficients, $b_i(\cdot,\cdot)$ are bivariate splines with degree $d$ and smoothness $r$, and $m_1, m_2$ are the dimensions of the bivariate spline spaces. The penalized least squares method of tensor product bivariate spline denoising is to find the spline coefficients $c_{ij}$ which solves
\begin{equation}
\label{PLS4D}
\min_{s\in {\cal S}} \sum_{i,j=1, \cdots, N} |s(x_i,y_i, u_j, v_j)- z_{ij}|^2  
+ \lambda {\cal E}_{2\times 2}(s)
\end{equation}
with $\lambda\approx 1$, and ${\cal E}_{2\times 2}(s)$ is
defined as
\begin{equation}
\begin{aligned}
\label{E4}
{\cal E}_{2\times 2}(s) &:=\int_{[0,1]^2}\left(\int_{[0,1]^2} \left|\frac{\partial^2}{\partial x^2} s\right|^2 + 2
\left|\frac{\partial^2}{\partial x \partial y} s\right|^2 + \left|\frac{\partial^2}{\partial y^2} s\right|^2 dxdy\right)dudv\\
&+\int_{[0,1]^2}\left(\int_{[0,1]^2} \left|\frac{\partial^2}{\partial u^2} s\right|^2 + 2
\left|\frac{\partial^2}{\partial u \partial v} s\right|^2 + \left|\frac{\partial^2}{\partial v^2} s\right|^2 dudv\right)dxdy. 
\end{aligned}  
\end{equation}

Let us explain next the computational procedure for finding the spline coefficients $c_{ij}$ based on a two-stage bivariate spline denoising scheme. 
Recall that tensor product splines for data interpolation were explained in \cite{deB79}. We extend its ideas to data denoising.  
For a given data set $\{(x_i,y_i, u_j,v_j,z_{i,j}), i, j=1, \cdots, N\}$ with data locations 
$(x_i,y_i)\in [0, 1]^2$ and $(u_j,v_j)\in [0, 1]^2$ and  noisy data values $z_{i,j}$, $i,j=1, \cdots, N$, we can write 
\begin{equation} \label{4dTPD}
    s(x,y,u,v)=\sum_{i=1}^{m_1}\sum_{j=1}^{m_2} c_{ij}b_{j}(u,v)b_{i}(x,y).
\end{equation}

Suppose our data is equally-spaced over $[0,1]^2\times[0,1]^2$, i.e., $N=m_1=m_2$. Let us denote $d_i(u,v)=\sum_{j=1}^{m_2} c_{ij}b_{j}(u,v)$, then we can write equation (\ref{4dTPD}) as
$s(x,y,u,v)=\sum_{i=1}^{m_1}d_i(u,v)b_i(x,y)$.
For fixed $(u_k,v_k)$, $k=1,\cdots,N$, write $d_i(u_k,v_k)=d_{ik}$ for all $i=1, \cdots, m_1$. For each fixed $k$, we can find the intermediate spline coefficients $d_{ik}$ via (\ref{PLS2D}) by letting 
\begin{equation} \label{4dTPD1}
    s(x_\ell,y_\ell)_k := s(x_\ell,y_\ell,u_k,v_k)=\sum_{i=1}^{m_1}d_{ik} b_i(x_\ell,y_\ell)
\end{equation}
for $\ell=1,\cdots,N$.
Once we have $d_{ik}$, then for each fixed $i$, we can find the spline coefficients $c_{ij}$ via (\ref{PLS2D}) by letting
\begin{equation} \label{4dTPD2}
    s(u_k,v_k):=d_{ik} = \sum_{j=1}^{m_2}c_{ij}b_j(u_k,v_k)
\end{equation}
for $k=1,\cdots,N$.

The advantage of tensor product spline denoising is its computational efficiency. If we directly solve the penalized least squares problem (\ref{PLS4D}) for the coefficients $c_{ij}$ without using this tensor product 
approach, then the matrix size associated in (\ref{4dTPD}) is $N^2\times m_1m_2$. Hence, solving it directly requires the computation complexity $\mathcal{O} (m_1^2 m_2^2 N^2)$. However, if we solve it by using tensor product via (\ref{4dTPD1}) and (\ref{4dTPD2}), then we only need to solve $N$ systems whose matrix size is of $N\times m_1$ and another $N$ systems whose matrix size is of $N\times m_2$. Therefore the computational complexity for solving them directly requires $\mathcal{O}(Nm_1^2 N+N m_2^2 N)=\mathcal{O}((m_1^2+m_2^2)N^2)$. If we use large degree $d$ and high smoothness $r\ge 1$ for denoising, then $m_1^2+m_2^2\ll m_1^2m_2^2$. Therefore, the computational cost for the two-stage tensor product denoising technique is much less than the direct denosing technique. This is why we adopt the tensor product spline denosing method. For the general case $d>4$, we can easily extend this idea to have a multi-stage denoising scheme. We leave out the details here.

\begin{remark}
    For any dimension $d>4$, it is not hard to see that $d$ can be written as a sum of $2$'s and $3$'s. Therefore, we can apply tensor products of bivariate and trivariate splines for denoising functions in any dimension. One may also consider the tensor products of univarite spline for denoising, however, the computation is much more demanding because the number of products needed in the univariate spline case is much larger than using bivariate and trivariate splines. 
\end{remark}

For each high dimensional linear KB-spline obtained via (\ref{KB}), we can apply such a computational scheme to solve (\ref{PLS4D}) and obtain the corresponding high dimensional linear LKB-spline, which is useful for approximation. We will use these linear LKB-splines as basis for high dimensional function approximation.

\section{Numerical experiments for LKB-splines based approximation } \label{experiments}

Let us present the numerical results for function approximation in $\mathbb{R}^d$ with $d=4$ and $d=6$ based on the linear LKB-spline basis obtained via the computational procedure described in the previous section.

We shall use discrete least squares (DLS) method to approximate any function $f\in C([0, 1]^d)$. Let $\{\mathbf{x}_i\}_{i=1}^N$ be a set of discrete points over $[0, 1]^d$. For any $f\in C([0,1]^d)$, we use the function values at these data locations to find an 
approximation $F_n= \sum_{j=1}^{dn} c_j^* LKB_{n,j}$ by DLS method 
where $c_j^*$, $j=1,\cdots,dn$, is the solution of the following minimization
\begin{equation}
\label{DLS}
\min_{ c_j} \| f- \sum_{j=1}^{dn} c_j LKB_{n,j}\|_{\cal P}.
\end{equation}
The notation $\|f\|_{\cal P}$ denotes the RMSE semi-norm based on the function values $f$ over 
these $N$ sampled data points in $[0, 1]^d$. We shall report the accuracy $\|f- F_n(f)\|_{\cal PP}$, 
where $\|f\|_{\cal PP}$ is the RMSE semi-norm based on more than $N$ function values. One important result established in \cite{LaiShenKST21}  for LKB-splines based approximation is the following.
\begin{theorem} [cf. \cite{LaiShenKST21}]
\label{newmain}
Suppose that $f\in C^2([0,1]^2)$. Let $F_n$ be the discrete least squares approximation of $f$ defined in (\ref{DLS}). Suppose that the points $\{x_i,y_i, f(x_i,y_i)+\epsilon_i)\}_{i=1}^N$ for 
(\ref{DLS}) are the same as the points for denoising KB-splines to have the LKB-splines. 
Then 
\begin{equation}
\label{newestimate}
\|f- F_n\|_{\cal P} \le  C\|f\|_{2,\infty} |\triangle|^2+ 2 \|\epsilon\|_{\cal P} + \frac{1}{\sqrt{N}} \sqrt{{\cal E}_2(f)}
\end{equation}
for a  positive constant $C$ independent of $f$ and triangulation $\triangle$.  
\end{theorem} 

In 4D, we sampled $11^4$ equally-spaced data across $[0,1]^4$ and fit a DLS approximation of a continuous function $f$ with 4D linear LKB-spline as basis, and we compute the RMSEs based on $26^4$ equally-spaced data across $[0,1]^4$. 
The following 10 testing functions across different families of continuous functions are used to check the approximation error.

\begin{eqnarray*} 
f_1 &=& (1+2x+3y+4u+5v)/15; \cr
f_2 &=& (x^2+y^2+u^2+v^2)/4; \cr
f_3 &=& (x^4+y^4+u^4+v^4)/4; \cr
f_4 &=& (\sin(x)\exp(y)+\cos(x)\exp(u)+\sin(x)\exp(v))/(3\exp(1)); \cr
f_5 &=& 1/(1+x^2+y^2+u^2+v^2); \cr 
f_6 &=& \sin(\pi x)\sin(\pi y)\sin(\pi u)\sin(\pi v); \cr
f_7 &=& (\sin(\pi(x^2+y^2+u^2+v^2))+1)/2;\cr 
f_{8} &=& \exp(-x^2-y^2-u^2-v^2);\cr
f_9 &=& \max(x-0.5)\max(y-0.5)\max(u-0.5)\max(v-0.5);\cr 
f_{10} &=& \max(x+y+u+v-2,0);
\end{eqnarray*}

In 6D, we sampled $6^6$ equally-spaced data across $[0,1]^6$ and fit a DLS approximation of a continuous function $f$ with 6D linear LKB-splines, and we compute the RMSEs based on $11^6$ equally-spaced data across $[0,1]^6$. The following 10 testing functions across different families of continuous functions are used to check the approximation error. 

\begin{eqnarray*} 
f_1 &=& (1+2x+3y+4z+5u+6v+7w)/28; \cr
f_2 &=& (x^2+y^2+z^2+u^2+v^2+w^2)/6; \cr 
f_3 &=& (x^3y^3+x^3z^3+y^3z^3+x^3u^3+u^3v^3+v^3w^3)/6; \cr
f_4 &=& (\sin(x)e^y+\cos(x)e^z+\sin(x)e^u+\cos(y)e^v+\sin(x)e^w)/(5e); \cr
f_5 &=& 1/(1+x^2+y^2+z^2+u^2+v^2+w^2); \cr 
f_6 &=& \sin(\pi x)\sin(\pi y)\sin(\pi z)\sin(\pi u)\sin(\pi v)\sin(\pi w); \cr
f_7 &=& (\sin(\pi(x^2+y^2+z^2+u^2+v^2+w^2))+1)/2; \cr 
f_8 &=& \exp(-x^2-y^2-z^2-u^2-v^2-w^2); \cr
f_9 &=& \max(x-0.5)\max(y-0.5)\max(z-0.5)\max(u-0.5)\max(v-0.5)\max(w-0.5);\cr 
f_{10} &=& \max(x+y+z+u+v+w-3,0);
\end{eqnarray*}

\begin{table}[t]
	\centering 
	\caption{RMSEs (computed based on $26^4$ equally-spaced locations) of the DLS fitting based $11^4$ equally-space sampled data and the DLS fitting based on pivotal locations in 4D, where $128, 241, 531$ are numbers of pivotal locations.} \label{tab1}
	
	\begin{tabular}{c|cc|cc|cc}
    \toprule
		& \multicolumn{2}{|c|}{$n=100$} & 
    \multicolumn{2}{|c|}{$n=300$} &
    \multicolumn{2}{c}{$n=1000$} \cr 
    \cmidrule{2-7}
  \# Sampled Data & $11^4$ & $128$ & $11^4$ & $241$ & $11^4$ & $531$ \cr
  \midrule
		$f_1$ & 3.06e-04 & 8.90e-04 & 6.02e-05 & 4.24e-04 & 2.79e-06 &  9.86e-06 \cr    
  $f_2$ & 9.70e-04 & 2.75e-03 & 4.35e-04 & 1.63e-03 & 2.66e-04 & 5.85e-04 \cr
  $f_3$ & 4.00e-03 & 1.13e-02 & 1.87e-03 & 6.88e-03 & 1.12e-03 & 2.32e-03 \cr 
  $f_4$ & 5.86e-04 & 1.88e-03 & 3.23e-04 & 1.45e-03 & 1.62e-04 & 4.31e-04 \cr 
  $f_5$ & 1.39e-03 & 3.63e-03 & 4.76e-04 & 1.80e-03 & 2.67e-04 & 7.07e-04 \cr 
  $f_6$ & 3.40e-02 & 1.07e-01 & 1.33e-02 & 7.80e-02 & 3.96e-03 & 2.24e-02 \cr 
  $f_7$ & 9.75e-02 & 3.07e-01 & 4.13e-02 & 1.96e-01 & 1.57e-02 &  5.40e-02 \cr 
  $f_8$ & 1.54e-03 & 3.78e-03 & 6.28e-04 & 2.55e-03 & 3.58e-04 & 8.85e-04 \cr 
  $f_9$ & 3.51e-04 & 1.29e-03 & 1.80e-04 & 1.56e-03 & 1.03e-04 & 5.59e-04 \cr 
  $f_{10}$ & 2.53e-02 & 5.32e-02 & 1.96e-02 & 8.25e-02 & 1.40e-02 & 3.45e-02 \cr 
  \bottomrule		
	\end{tabular}
\end{table}   

	
	
		

\begin{table}[t]
	\centering 
	\caption{RMSEs (computed based on $11^6$ equally-spaced locations) of the DLS fitting based $6^6$ equally-space sampled data and the DLS fitting based pivotal locations in 6D, where $13, 24, 70$ are the numbers of pivotal locations.} \label{tab2}
	
	\begin{tabular}{c|cc|cc|cc}
    \toprule
		& \multicolumn{2}{|c|}{$n=20$} & 
    \multicolumn{2}{|c|}{$n=40$} &
    \multicolumn{2}{c}{$n=120$} \cr 
    \cmidrule{2-7}
  \# Sampled Data & $6^6$ & $13$ & $6^6$ & $24$ & $6^6$ & $70$ \cr
  \midrule
		$f_1$ & 5.09e-02 & 7.81e-02 & 3.03e-02 & 5.52e-02 & 7.78e-03 & 3.61e-02 \cr    
  $f_2$ & 4.56e-02 & 1.31e-01 & 4.09e-02 & 8.30e-02 & 1.73e-02 & 5.31e-02 \cr
  $f_3$ & 9.70e-02 & 1.29e-01 & 5.26e-02 & 8.39e-02 & 2.85e-02 & 5.09e-02 \cr 
  $f_4$ & 7.50e-02 & 2.50e-01 & 7.27e-02 & 1.50e-01 & 3.59e-02 & 1.23e-01 \cr 
  $f_5$ & 5.44e-02 & 1.88e-01 & 4.98e-02 & 1.22e-01 & 2.72e-02 & 7.25e-02 \cr 
  $f_6$ & 3.95e-02 & 8.07e-02 & 3.55e-02 & 6.14e-02 & 1.64e-02 & 4.45e-02 \cr 
  $f_7$ & 2.50e-02 & 8.71e-02 & 2.40e-02 & 6.46e-02 & 9.08e-03 & 3.94e-02 \cr 
  $f_8$ & 8.84e-02 & 9.39e-02 & 6.83e-02 & 7.39e-02 & 4.62e-02 & 5.40e-02 \cr 
  $f_9$ & 3.47e-01 & 3.55e-01 & 2.30e-01 & 2.56e-01 & 1.04e-01 & 1.86e-01 \cr 
  $f_{10}$ & 3.12e-02 &  7.66e-02 & 2.37e-02   & 5.65e-02 & 1.25e-02  & 3.94e-02 \cr 
  \bottomrule		
	\end{tabular}
\end{table}

In addition, we noticed that  the linear system associated with the DLS approximation has many zero or near zero columns due to the structure inner functions. 
As discussed in \cite{LaiShenKST21,S24}, we adopt the matrix cross approximation in \cite{ALS} to find the pivotal point set.  Based on the function values at the pivotal points in 
$[0, 1]^d$, we can simply solve the subsystem with much smaller size to find the approximation of $f$. Similar RMSEs are obtained and presented in Table~\ref{tab1} and \ref{tab2} side by side to show
that the approximation of $f$ for both approaches works well. More importantly, the approximation based on the function values over pivotal locations is similar to the approximation 
based on the $11^4$ equal-spaced function values. 

We plot the approximation error based on the pivotal location against $n$ on the log-log scale, hence the exponent of $n$ in the approximation rate is associated with the slope of the fitted line via linear regression. The results are shown in Figure~\ref{ConvgPlots}. For example, if a fitted line with slope approximately equal to $-0.5$, that indicates the approximation rate is about $\mathcal{O}(1/n^{1/2})$. This gives us a way to numerically estimate the exponent $\beta$ such that the outer function $g_f\in C^{0,\beta}([0,d])$. 

We also plot the number of pivotal locations needed to achieve those approximation errors, the results are shown in Figure~\ref{NNrPlots}. It shows that we only need fewer than $\mathcal{O}(nd)$ function values of $f$ to achieve the convergence rate $\mathcal{O}(1/n^{\beta})$ for some $\beta\in (0,1)$, or $\mathcal{O}(1/n)$, or $\mathcal{O}(1/n^2)$ based on the smoothness of outer function $g_f$.

\vspace{2mm}
\begin{remark}
    In general, there is no easy way to determine in theory which $\beta$ such that $g_f\in C^{0,\beta}([0,d])$. Instead, our approach can serve as a numerical approach to estimate the Hölder exponent $\beta$ for the outer function $g_f$.
\end{remark}

\vspace{2mm}

\begin{remark}
The primary computational burden for the results in Tables~\ref{tab1} and \ref{tab2} lies in the denoising step, where KB-splines are transformed into LKB-splines. This process requires a large number of data points and values due to the pervasive presence of noise over $[0, 1]^d$. As the dimensionality $d > 2$ increases, the number of required points and KB-spline values grows exponentially—such as in the tensor product spline method for denoising—resulting in computational costs that suffer from the curse of dimensionality. However, this denoising step can be pre-computed once for all and can be performed in parallel, significantly mitigating its impact on runtime. Once the LKB-splines are obtained, we build up the least squares data fitting matrix $X_{data}$ which again can be done in parallel. Next we sort out the determinant of submatrices from $X_{data}$ to form a matrix cross
approximation based on the row index set $I$ which corresponds to the pivotal locations. Then 
the remaining computational cost is limited to solving a least squares problem based on the pivoting
locations and the associated function values, which is far less demanding.
\end{remark}

\begin{figure}[t] 
    \centering
    \begin{tabular}{cc}

    \scriptsize{$f=(1+2x+3y+4u+5v)/15$} & \scriptsize{$f=(1+2x+3y+4z+5u+6v+7w)/28$}  \\
   
    \includegraphics[width=0.45\textwidth]{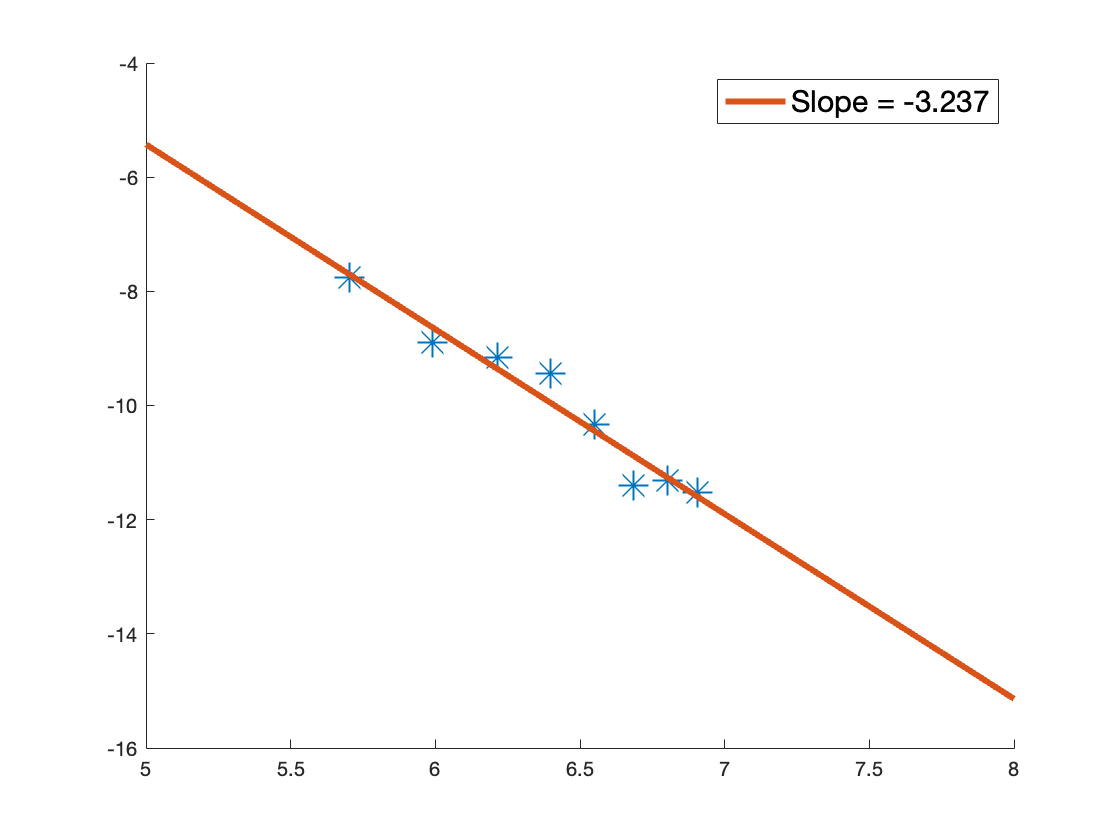} &
    \includegraphics[width=0.45\textwidth]{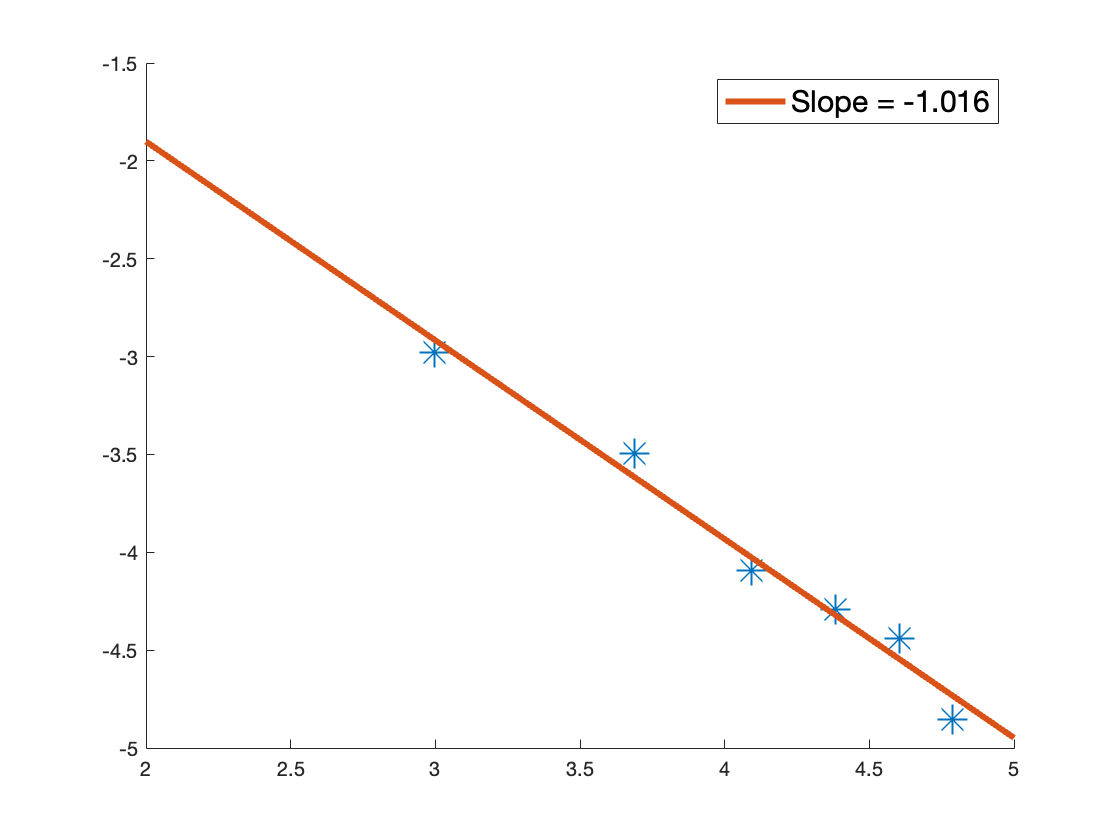}  \\

    \scriptsize{$f=\sin(\pi x)\sin(\pi y)\sin(\pi u)\sin(\pi v)$}& \scriptsize{$f=\sin(\pi x)\sin(\pi y)\sin(\pi z)\sin(\pi u)\sin(\pi v)\sin(\pi w)$} \\
    \includegraphics[width=0.45\textwidth]{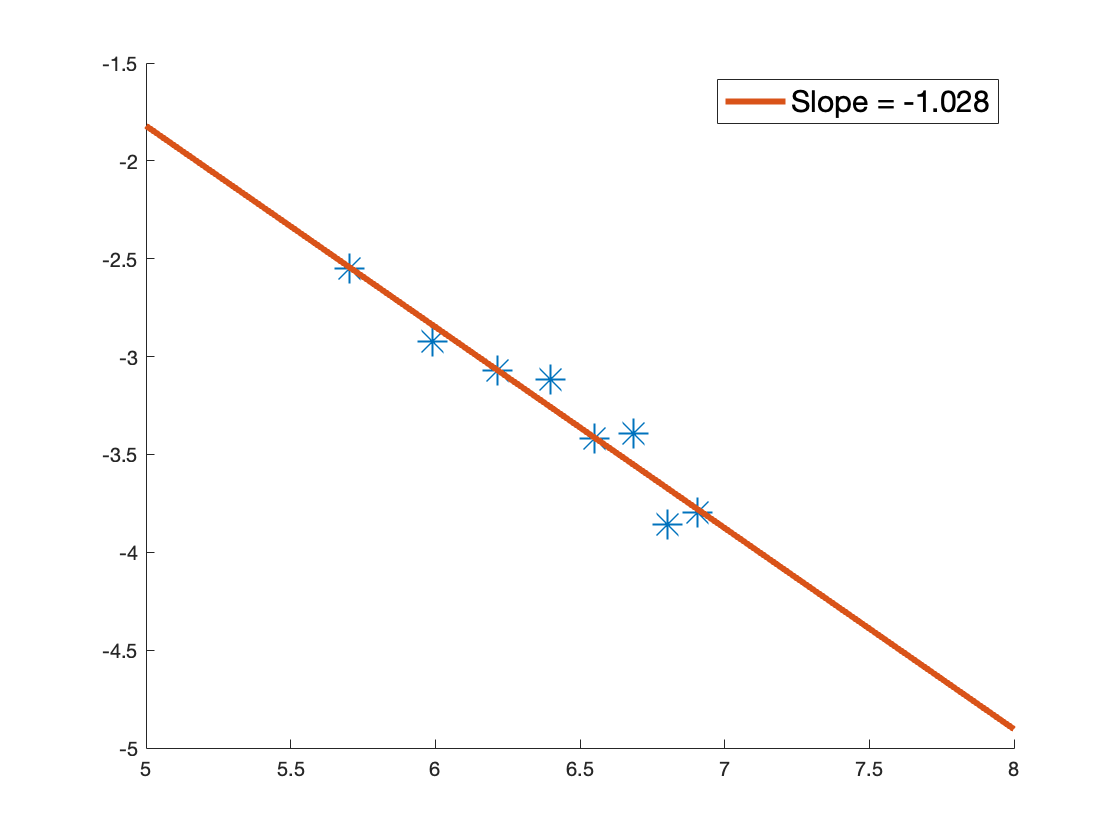} &
    \includegraphics[width=0.45\textwidth]{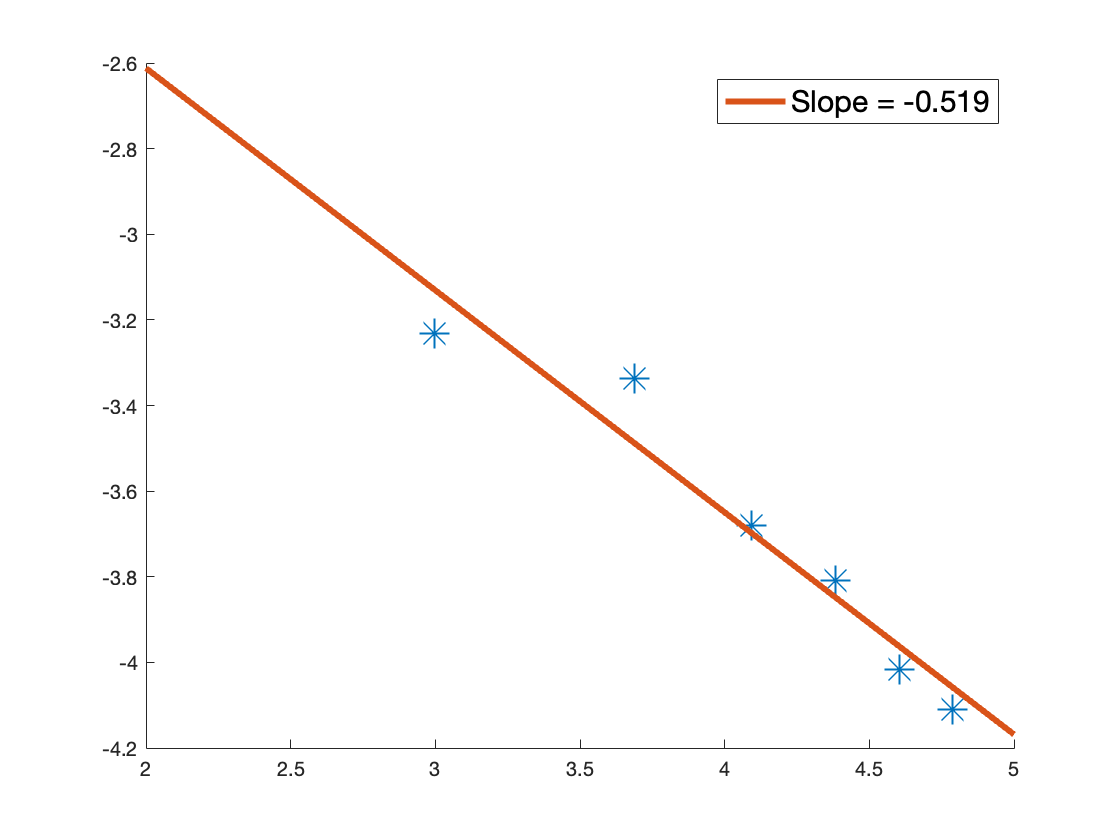} \\

    \scriptsize{$f=\max(x+y+u+v-2,0)$} & \scriptsize{$f=\max(x+y+z+u+v+w-3,0)$} \\
    \includegraphics[width=0.45\textwidth]{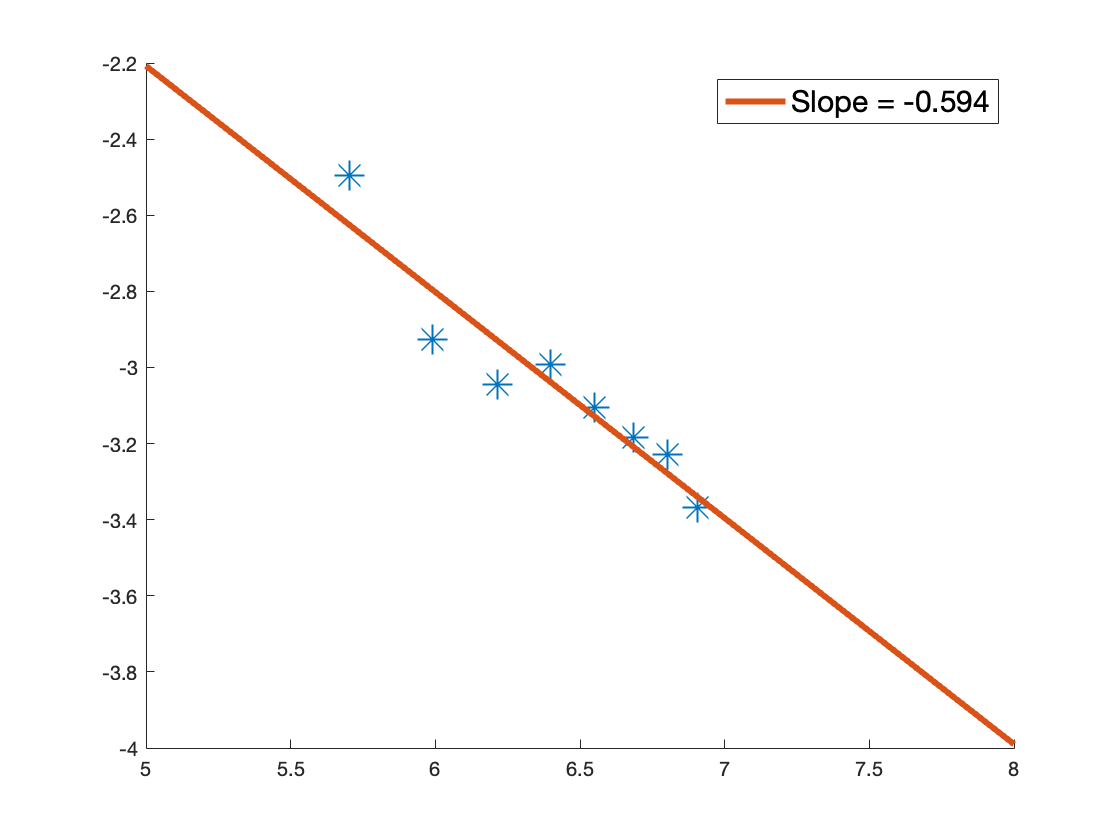} &
    \includegraphics[width=0.45\textwidth]{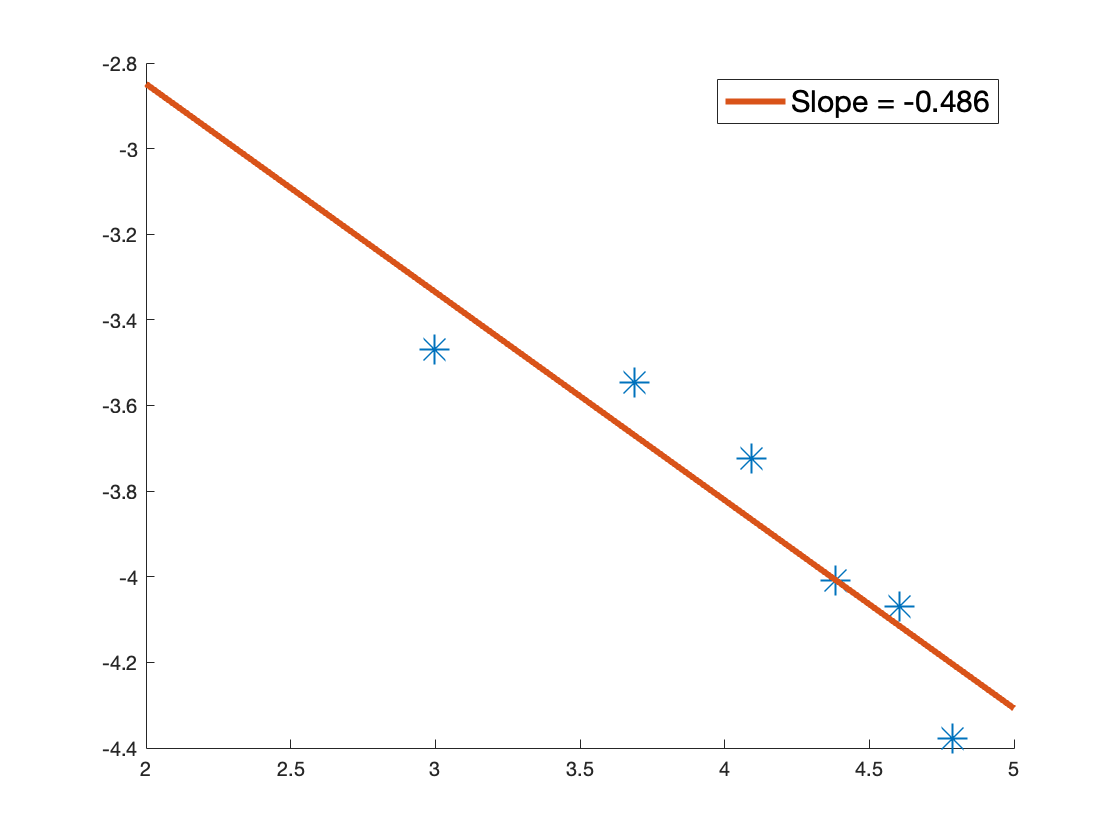} 
    \end{tabular}
    
    \caption{Plots of convergence rate on the Log-log scale in 4D and 6D based on pivotal dataset. \label{ConvgPlots}}
\end{figure} 

\begin{figure}[t] 
    \centering
    \begin{tabular}{cc}
    \includegraphics[width=0.42\textwidth]{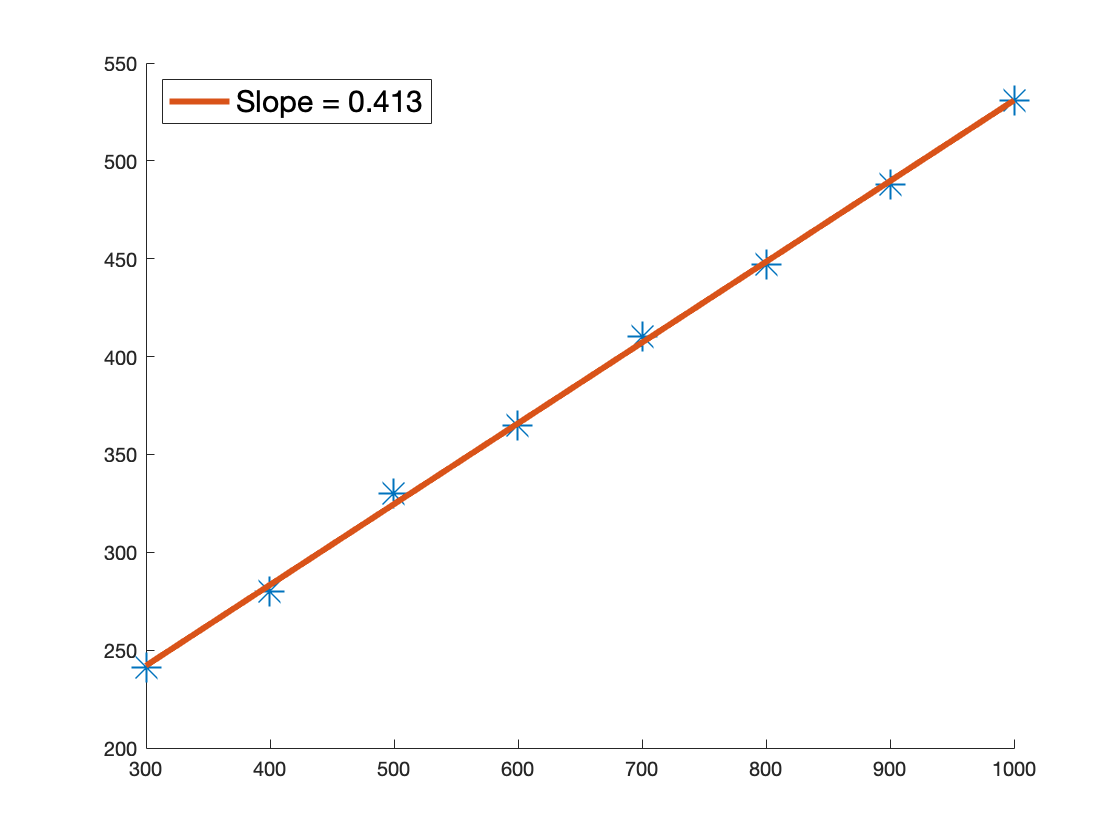} &
    \includegraphics[width=0.42\textwidth]{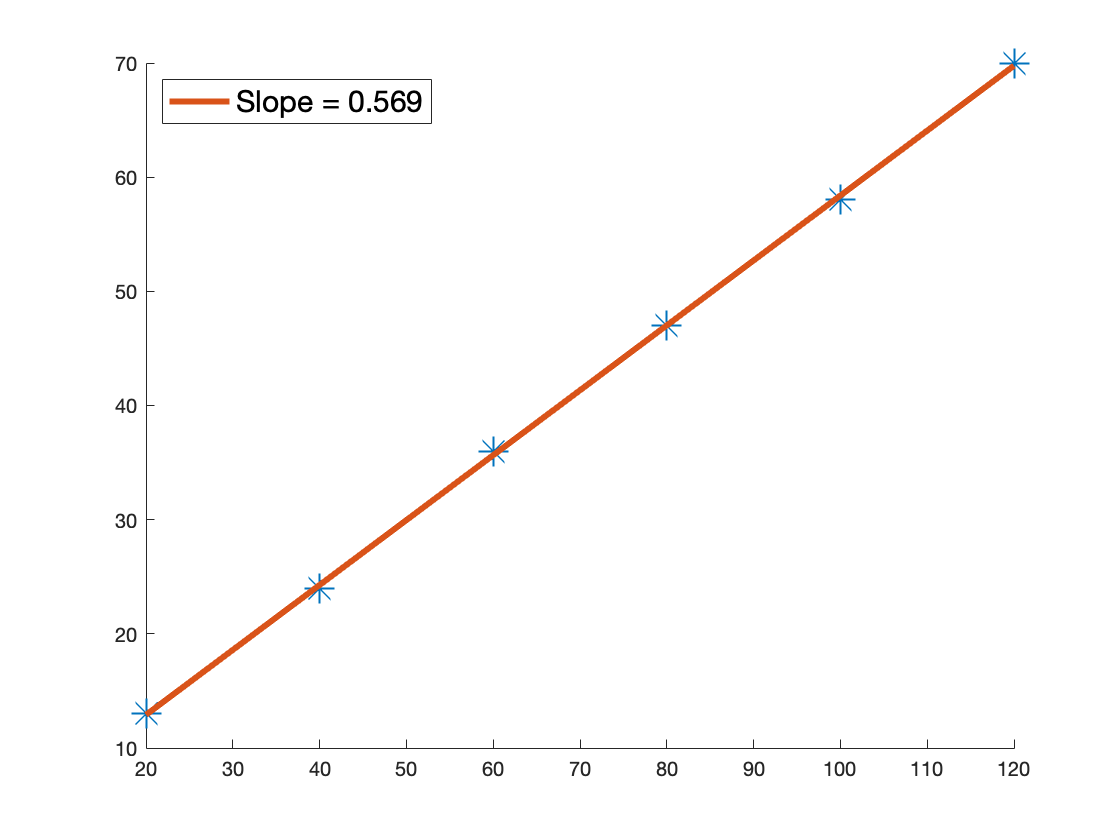}
    \end{tabular}
    
    \caption{Number of pivotal locations (vertical axis) against $n$ (horizontal axis) in 4D (left) and in 
    6D (right). \label{NNrPlots}}
\end{figure}

\section{Application to numerically solving the Poisson equation}
\label{secPDE}
Another powerful aspect of the LKB-splines based approximation scheme is that we can use it to solve partial differential equations. To start with, let
us consider the Poisson equation:
\begin{equation}
\label{PDE}
\left\{
\begin{array}{lll}
\Delta u &=& f, \bfx\in \Omega\cr
       u &=& 0, \bfx\in \partial \Omega.
\end{array}
\right.
\end{equation}
For simplicity, let us consider the 2D case where $\Omega=[0, 1]^2$. Let us use $f_i= LKB_{n,i}, i=1, \cdots, 2n$, as the right-hand side of (\ref{PDE}). 
We can use bivariate spline function of degree $d=8$ and $r=2$ to solve (\ref{PDE}) by using the spline collocation 
method as proposed in \cite{LL22} to obtain the solution $u_i$ 
$i=1, \cdots, 2n$. These form a set of basic Poisson solutions.  

Then for any $f$, we use $LKB_{n,i}, i=1, \cdots, 2n$, to 
approximate $f$. As discussed in the previous section, we can use $LKB_{n,i}$ to approximate $f$. Suppose
$\hat{f} = \sum_{i=1}^{2n} c_i(f) LKB_{n,i}$ is a good approximation of $f$, then the solution $u$ of the Poisson equation (\ref{PDE})
can be well approximated by  
\begin{eqnarray}
\label{ourmethod}
u_n= \sum_{i=1}^{2n} c_i(f) u_i.
\end{eqnarray}

To show $u-u_n$ goes to $0$ when $n \to \infty$, we consider $\|\Delta (u- u_n)\|_{L^2(\Omega)}$ 
which is $\| f- \sum_{i=1}^{2n} c_i(f) LKB_{n,i}\|_{L^2(\Omega)}$. Let us define a new norm $\|u\|_L$ on $H^2(\Omega)\cap H^1_0(\Omega)$ for the Poisson equation as: 
\begin{equation}
	\label{H2norm2}
	\|u\|_{L}= \|\Delta u\|_{L^2(\Omega)}.
\end{equation}
Lai and Lee in \cite{LL22} showed that the new norm is equivalent to the standard norm on 
Banach space $H^2(\Omega)\cap H^1_0(\Omega)$.  In particular, there exists some constant $A>0$ such that $A \|u \|_{H^2(\Omega)} \le \|u \|_L$. 

Now letting $\epsilon_1=  \| f- \sum_{i=1}^{2n} c_i(f) LKB_{n,i}\|_{L^2(\Omega)}$, 
The above inequality shows that
$$
\|u- u_n\|_{H^2(\Omega)}\le \epsilon_1/A.
$$    
As shown in Theorem~\ref{newmain}, $\epsilon_1\to 0$ if $n,N\to\infty$ and the triangulation size $|\triangle|\to 0$. Under some regularity assumption on $u$, one can further show 
\begin{eqnarray*}
		\|u-u_n\|_{L^2(\Omega)} \le C|\triangle|^2 \epsilon_1 \hbox{ and }
		\|\nabla (u-u_n)\|_{L^2(\Omega)} \le C|\triangle| \epsilon_1
\end{eqnarray*}
where $C=1/A$. We refer the interested readers to \cite{LL22} for more details.


\subsection{Numerical results}
For numerical experiments, we will use the following six functions as testing functions for the solution
of (\ref{PDE}). Their right-hand-side $f$ can be easily computed. 
\begin{eqnarray*}
u_1 &=& x(1-x)y(1-y)/4; \cr
u_2 &=& \sin(\pi x)\sin(\pi y); \cr
u_3 &=& \sin(x)(1-x)(1-y)\sin(y); \cr
u_4 & = &(x(1-x)y(1-y)/4)^2; \cr
u_5 & = & (\sin(\pi x)\sin(\pi y))^2; \cr
u_6 &=&  (\sin(x)(1-x)(1-y)\sin(y))^2; 
\end{eqnarray*}
We first use the linear LKB-splines to approximate the right-hand-side $f$ associated with $u_1,\cdots,u_6$ over $[0, 1]^2$. We sampled $101^2$ equally-spaced data across $[0,1]^2$ and fit a DLS approximation of a continuous function $f$ with LKB-splines as basis, and we check the RMSEs based on $1001^2$ equally-spaced data across $[0,1]^2$. See Table~\ref{tab3} for the numerical results.  


\begin{table}[t]
	\centering 
	\caption{RMSEs (computed based on $1001^2$ equally-spaced locations) of the approximation for the right-hand-side function $f=\Delta u$ based on equally-space sampled data and based on $59, 76, 136$ pivotal locations.} \label{tab3}
	
	\begin{tabular}{c|cc|cc|cc}
    \toprule
		& \multicolumn{2}{|c|}{$n=100$} & 
    \multicolumn{2}{|c|}{$n=300$} &
    \multicolumn{2}{c}{$n=1000$} \cr 
    \cmidrule{2-7}
  \# Sampled Data & $101^2$ & $59$ & $101^2$ & $76$ & $101^2$ & $136$ \cr
  \midrule
		$\Delta u_1$ & 4.90e-04 & 9.67e-04 & 2.46e-04 & 5.39e-04 & 1.01e-04 & 2.91e-04 \cr    
  $\Delta u_2$ & 3.04e-02 & 4.35e-02 & 1.31e-02 & 2.22e-02 & 3.90e-03 & 6.98e-03 \cr
  $\Delta u_3$ & 2.00e-03 & 3.80e-03 & 1.00e-03 & 2.30e-03 & 3.77e-04 & 1.10e-03 \cr 
  $\Delta u_4$ & 9.05e-05 & 1.32e-04 & 3.85e-05 & 8.59e-05 & 6.98e-06 & 1.86e-05 \cr 
  $\Delta u_5$ & 2.38e-01 & 4.26e-01 & 1.13e-01 & 1.53e-01 & 2.83e-02 & 6.06e-02 \cr 
  $\Delta u_6$ & 1.50e-03 & 2.20e-03 & 4.90e-04 & 9.49e-04 & 1.20e-04 & 3.17e-04 \cr 
  \bottomrule		
	\end{tabular}
\end{table}


Next, we compute the spline solution of
the Poisson equation for each LKB-spline as the right-hand side of the PDE (\ref{PDE}) 
to obtain $u_i$'s. 
As explained above, we use the coefficients of linear LKB-spline approximation of each right-hand-side function $f$ to form the solution of the Poisson equation.   We apply the method described
in (\ref{ourmethod}) to approximate the solution of the Poisson equation and the numerical results are shown
in Table~\ref{tab4}.  
Similarly, one can use LKB-splines to approximate the solution of the Poisson equation in 3D, we leave it to the future work.

So far we only explained how to use LKB-splines for approximating the solution of the Poisson equation with zero boundary conditions. The underlying domain of interest is $[0, 1]^2$. This is the  
simplest PDE.  We are currently investigating how to use the LKB-splines for the numerical solution of the Poisson equation over arbitrary polygons with nonzero Dirichlet boundary conditions.


\begin{table}[t]
	\centering 
	\caption{RMSEs (computed based on $1001^2$ equally-spaced locations) of the approximation for the true solution $u$ based on equally-spaced sampled data and based on $59, 76, 136$ pivotal locations.} \label{tab4}
	
	\begin{tabular}{c|cc|cc|cc}
    \toprule
		& \multicolumn{2}{|c|}{$n=100$} & 
    \multicolumn{2}{|c|}{$n=300$} &
    \multicolumn{2}{c}{$n=1000$} \cr 
    \cmidrule{2-7}
  \# Sampled Data & $101^2$ & $59$ & $101^2$ & $76$ & $101^2$ & $136$ \cr
  \midrule
		$u_1$ & 1.04e-06 & 1.57e-05 & 4.02e-07 & 5.97e-06 & 8.09e-08 & 1.92e-06 \cr    
  $u_2$ & 1.82e-04 & 1.20e-03 & 5.09e-05 & 6.95e-04 & 9.53e-06 & 1.05e-04 \cr
  $u_3$ & 4.56e-06 & 6.30e-05 & 1.82e-06 & 2.38e-05 & 3.23e-07 & 6.47e-06 \cr 
  $u_4$ & 2.11e-07  & 6.91e-07 & 6.71e-08 & 1.07e-06 & 8.34e-09 & 1.05e-07 \cr 
  $u_5$ & 1.80e-03 & 8.40e-03 & 5.98e-04 & 1.90e-03 & 1.26e-04 & 1.01e-03 \cr 
  $u_6$ & 4.30e-06 & 1.15e-05 & 9.51e-07 & 9.51e-06 & 1.50e-07 & 1.77e-06 \cr 
  \bottomrule		
	\end{tabular}
\end{table}   


The advantage of this approach is that 
the basic solutions $u_i$'s of the 
Poisson equation can be solved beforehand and stored and one only needs to approximate
the right-hand-side function. Note that
the right-hand-side function $f$ can be easily approximated
based on the pivotal point locations without using a large amount of the function values if $f$ is Kolmogorov-H\"older continuous. This approach provides an efficient method for solving PDE numerically.  When 
$f$ is Kolmogorov-H\"older continuous, our method based on LKB-splines will give a very good approximation of the right-hand-side function $f$ and hence, the solution of the Poisson equation as demonstrated in this section.

\section{Conclusion and future research}\label{secCon}

In this paper, we propose a novel approach for approximating $f\in C([0,1]^d)$ based on linear spline approximation of the outer function via Kolmogorov superposition theorem. By employing linear KB-spline approximation for the outer function, we demonstrate that the optimal approximation rate of $\mathcal{O}(1/n^2)$ can be achieved with the complexity $O(nd)$ for a dense subclass of continuous functions whose outer function is twice continuously differentiable. Additionally, the approximation constant for that subclass grows linearly with  $d$. Moreover, we introduce a tensor product spline denoising technique to obtain LKB-splines, enabling accurate function approximation in dimension $d=4$ and $d=6$. We further validate numerically that, for these dimensions, the proposed LKB-spline-based approach achieves the expected approximation rate and complexity. Finally, we apply this approximation scheme to solve the Poisson equation numerically, achieving highly accurate results.

Future research can be pursued along several promising directions. First, a complete characterization of the relationship between the original function and the outer function remains challenging, and we aim to further explore this aspect. Second, highly oscillatory functions, such as high-frequency trigonometric functions, are difficult to approximate effectively using KB-splines and LKB-splines. To address this, we plan to investigate the use of Fourier basis to construct Kolmogorov-Fourier functions. Third, we seek to extend this approach to the approximation of discontinuous functions. If successful, this could open the door to analyzing a wide range of real-world applications, such as image and signal processing, from a fresh perspective.

\backmatter

\bmhead{Supplementary information}

We have provided a sample of Matlab code as supplementary material to support our numerical results. Other parts of code are available upon request.




\bmhead{Acknowledgments}
The authors are very grateful to the editor and referees for their constructive comments to improve the quality of this paper.

\section*{Declarations}


\begin{itemize}
\item \textbf{Funding.} The first author is supported by the Simons Foundation for collaboration grant \#864439. 
\item \textbf{Competing interests.} On behalf of all authors, the corresponding author states that there is no conflict of interest.
\end{itemize}

\bibliographystyle{plain}
\bibliography{sn-bibliography}

\end{document}